\documentclass[english,11pt]{amsart}
\input{preambule_en.sty}
\title[Hölder Regularity of Distributional Volume Forms]{\bf \textsc{Hölder Regularity of \\ Distributional Volume Forms}}
\author{Thomas Jaffard}
\address{Sorbonne Université, Université Paris Cité, CNRS, Laboratoire de Probabilités, Statistique et Modélisation, LPSM, F-75005, France}
\date{\today}
\subjclass[2020]{Primary 46F10; Secondary 49Q15; 42C40.}
\keywords{Hölder regularity, distributional differential forms, Besov spaces, wavelets, fractional dimension}
\email{jaffard@lpsm.paris}
\begin{document}

\begin{abstract}
Let $f, g^1, \dots, g^d : \R^d \longrightarrow \R$ be Hölder continuous functions. If the Hölder exponents of these functions are less than $1$ but sufficiently large, we use the integral introduced by Züst in \cite{Zu11} to construct  a distribution, denoted by
$\fdgd,$ which depends continuously on the functions $f, g^1, \dots, g^d$ in a sense that we shall  specify, and which coincides with the function $f\det(\dg)$
when the functions $g^i$ are Lipschitz.
We show that this distribution  is entirely characterized by these properties and determine its Hölder regularity.

We use this distribution to define the integral $
\int_{\Omega} \fdgd$
by duality, for general domains $\Omega \subset \R^d$.
When $\Omega$ is a rectangle, this integral coincides with Züst’s construction. 
We then establish a new criterion on the domain $\Omega$ ensuring that the integral is well defined.
This criterion allows to recover a condition of Bouafia \citep{BOU24}  on the perimeter of the domain, and in the case when $d = 2$, the condition of Alberti–Stepanov–Trevisan \citep{AlSTTRE24} on the upper box dimension of the boundary.
\end{abstract}

\maketitle
\setcounter{tocdepth}{1}
\tableofcontents

\section{Introduction}

\subsection{Presentation of the problem and Züst's integral}

Let $d \ge 1$ be an integer. We consider real-valued functions $f, g^1, \dots, g^d$ defined on $\R^d$, and we denote by $g$ the function 
$$
\begin{array}{rrcl}
g : & \R^d & \longrightarrow & \R^d  \\
& x & \longmapsto & (g^1(x), \dots, g^d(x)).
\end{array}
$$
Assuming that the function $f$ is continuous and that the functions $g^i$ are of class $\C^1$, the quantity $\fdgd$ is a continuous $d$-form on $\R^d$, defined by
$$
\fdgd = f\det(\dg) \dx^1 \wedge \dots \wedge \dx^d.
$$
The integral of this differential form over the unit cube in $\R^d$ is given by  
\begin{equation}
\label{eq : intégrale sur [0,1]^d de fdg1^..^dgd dans le cas lisse}
\int_{[0,1]^d} \fdgd = \int_{[0,1]^d} f(x)\det(\dg(x)) \dx,
\end{equation} 
and it can be integrated over any $d$-cell $c : [0,1]^d \longrightarrow \R^d$ of class $\C^1$ by pulling it back:
$$
\int_c \fdgd = \int_{[0,1]^d} (f\circ c)\d(g^1\circ c) \wedge \dots \wedge \d(g^d\circ c).
$$
We can then, for instance, extend this integral into a cochain defined on the set of $\C^1$ $d$-chains in $\R^d$, or, in more general geometric contexts, consider the case where the functions are defined on an oriented $\C^1$ manifold of dimension $d$.

Adopting now a distributional point of view on the object $\fdgd$, we can identify it with the continuous function $f\det(\dg)$ and use Lebesgue integration theory to give meaning to the integrals
$$
\int_{\Omega} f(x)\det(\dg(x)) \dx, \quad \Omega \text{ bounded Borel subset of } \R^d.
$$
Note that it is still possible to make sense of these integrals under weaker regularity assumptions. Indeed, if the functions $g^i$ are only assumed to be locally Lipschitz, a theorem of Lebesgue asserts that the function $g$ is differentiable almost everywhere and that its Jacobian is locally bounded. It is then sufficient to assume that the function $f$ is locally integrable for these integrals to be well-defined.

When the function $g$ is no longer differentiable, it is not clear to give a meaning to the distribution $\fdgd$, or even to $\dgd$. In \citep{BRNG11}, Brezis and Nguyen study this problem and provide optimal results for defining the distribution 
$$ \det(\nabla g)  = \det(\dg) = \dg^1 \wedge \dots \wedge \dg^d $$
on a regular open bounded subset of $\R^d$, 
when the function $g$ belongs to a fractional Sobolev space. To establish their results, they show that the mapping $g \mapsto \det(\nabla g)$ is continuous with respect to a Sobolev norm on the space $\C^1(\overline\Omega)$ and extend this mapping by density.
In \cite{Zu11}, Züst also addresses this problem but adopts a different approach. Rather than constructing a distribution $\fdgd$, the integrals
$$\int \fdgd$$
 are defined directly when
the functions $f, g^1, \dots, g^d$ are Hölder continuous but not Lipschitz. If the sum of their Hölder exponents is sufficiently large, he manages to define the above integral over a rectangle, or more generally over the image of a rectangle under a Lipschitz mapping. We recall below the part of his work we will need in the sequel.

Let $\alpha, \beta_1, \dots, \beta_d$ be real numbers satisfying
\begin{equation}
\label{eq : alpha + beta > d}
\alpha, \beta_1, \dots, \beta_d \in (0,1) \quad \text{and} \quad \alpha + \sum_{i=1}^d \beta_i > d.
\end{equation}
Let $a = (a_1, \dots, a_d)$ and $b = (b_1, \dots, b_d)$ in $\R^d$ such that $a_i \le b_i$ for each $i \in \{1, \dots, d\}$.
 We denote by $R$ the rectangle $\prod_{i=1}^d [a_i, b_i]$. 
The space $\C^{\alpha}(R)$ consists of real-valued functions $f$ on the rectangle $R$ that are $\alpha$-Hölder continuous, i.e. such that the quantity
$$
\|f\|_{\alpha,R} = \sup_{\substack{x,y \in R \\ x \neq y}}
\;\frac{|f(y) - f(x)|}{|y - x|^{\alpha}}
$$
is finite. Recall that for any real number $\alpha' \in (0,\alpha]$, we have the continuous embedding
 $$\C^{\alpha}(R) \hookrightarrow C^{\alpha'}(R).$$
For every $(d+1)$-tuple of functions 
$$ (f, g^1, \dots, g^d) \in \Calpha(R) \times \C^{\beta_1}(R) \times \dots \times \C^{\beta_d}(R),$$
Züst constructs the integral
$$ \int_R \fdgd $$
by means of the discrete approximation
$$ \int_R \fdgd \approx f(a_1, \dots, a_d) \int_R \dg^1 \wedge \dots \wedge \dg^d, $$
inspired by Stokes’ formula
$$ \int_R \dg^1 \wedge \dots \wedge \dg^d = \int_{\partial R} g^1 \dg^2 \wedge \cdots \wedge \dg^d, $$
where the right-hand side is defined as a signed sum of $2^{d+1}$ well-defined integrals in dimension $d-1.$
More precisely, he proves the following theorem.

\begin{theo}[Züst]
\label{theo : intégrale de Zust}
Under assumption \eqref{eq : alpha + beta > d} on the exponents $\alpha, \beta_1, \dots, \beta_d$, there exists a unique $(d+1)$-linear mapping
$$
\begin{array}{rrcl}
\displaystyle\int_{R} : & \Calpha(R) \times \C^{\beta_1}(R) \times \dots \times \C^{\beta_d}(R) & \longrightarrow & \R  \\
& (f, g^1, \dots, g^d) & \longmapsto & \displaystyle\int_{R} f \dg^1 \wedge \dots \wedge \dg^d
\end{array}
$$
satisfying the following properties:
\begin{enumerate}[label=\arabic*.]
\item If the functions $g^1, \dots, g^d$ are Lipschitz on $R$, then 
\[
\int_{R} \fdgd = \int_{R} f(x)\,\det(\dg(x)) \dx.
\]
\item The mapping $\int_{R}$ is continuous in the following sense: for all sequences of functions $(f_n)_{n\in\N}$ and $(g^i_n)_{n\in\N}$, $i \in \{1,\dots,d\}$, defined on $R$ such that
\begin{enumerate}
\item the sequences $(f_n)_{n\in\N}$ and $(g^i_n)_{n\in\N}$ converge uniformly to $f$ and $g^i$, respectively;
\item $\sup_{n \in \N} \big(\|f_n\|_{\alpha} + \|g^i_n\|_{\beta_i}\big) < \infty$,
\end{enumerate}
we have
\[
\lim_{n\to \infty} \int_{R} f_n \dg^1_n \wedge \dots \wedge \dg^d_n
= \int_{R} \fdgd.
\]
\end{enumerate}
\end{theo}

The uniqueness of the mapping $\int_{R}$ follows from a general result in metric spaces which allows us to approximate uniformly any Hölder function by a sequence of Lipschitz functions bounded for the Hölder norm considered. We derive the following result from this property. Let us temporarily denote by $\int_{R, \alpha, \beta_1, \dots, \beta_d}$ the mapping defined in the framework of Theorem  \ref{theo : intégrale de Zust}. For all real numbers 
$$\alpha' \in (0, \alpha], \beta_1'\in (0, \beta_1], \dots, \beta_d'\in (0, \beta_d]$$ 
satisfying \eqref{eq : alpha + beta > d}, the restriction of 
${\int_{R,\alpha', \beta_1', \dots, \beta_d'}}$ to the space $\Calpha(R) \times \C^{\beta_1}(R) \times \dots \times \C^{\beta_d}(R)$ coincides with the mapping $\int_{R,\alpha, \beta_1, \dots, \beta_d}.$
This property will allow us to talk freely about the integral $\int_{R} \fdgd$ for all real-valued Hölder functions $f, g^1, \dots, g^d$ defined on $R$, as long as their Hölder exponents satisfy the condition \eqref{eq : alpha + beta > d}. With this clarification, we can now state the general sewing inequality proved by Züst during the construction of his integral.
For all real numbers $\alpha' \in (0,\alpha], \beta_1' \in (0,\beta_1], \dots, \beta_d' \in (0,\beta_d]$ satisfying \eqref{eq : alpha + beta > d}, the following estimate holds:
\begin{equation}
\label{eq : inégalité de couture}
\left| \int_{R} \fdgd - f(a_1,\dots,a_d) \int_{R} \dgd \right|
\le C\, \delta(R)^{\alpha' + \beta'}\, \|f\|_{\alpha',R} \prod_{i=1}^d \|g^i\|_{\beta_i',R},
\end{equation}
where 
$$ 
\beta' = \sum_{i=1}^d \beta_i' \quad \text{and} \quad \delta(R) = \max_{1 \le i \le d} (b_i - a_i),
$$ and the constant $C$ depends only on $d, \alpha', \beta_1', \dots, \beta_d'$.

The approximation result recalled above allows Züst to deduce from the smooth case, i.e. the case where the functions $g^i$ are Lipschitz, several natural properties satisfied by his integral. These include, for example, the antisymmetry of the integral with respect to the functions $g^i$, the fact that the integral vanishes if one of the $g^i$ is constant on the support of $f$, and the following additivity property which we will need to define our distribution.

\begin{prop}[Additivity over rectangles]
\label{prop : additivité sur les rectangles}
Let $P, Q \subset \R^d$ be two rectangles with disjoint interiors whose union $R = P \cup Q$ is a rectangle. Then,
$$\int_R \fdgd = \int_P \fdgd + \int_Q \fdgd.$$
\end{prop}
After showing that his integral is invariant under bi-Lipschitz reparametrization, Züst extends it to the case where the functions are defined on an oriented Lipschitz manifold of dimension $d.$ 
He then generalizes his previous constructions and shows that, under suitable assumptions, any current defined on a locally compact metric space extends uniquely to a mapping defined on a product of Hölder function spaces, provided that condition \eqref{eq : alpha + beta > d} is satisfied.

In dimension $1$, Züst's integral coincides with Young's integral \cite{YOU36,KON37}, which has regained interest with the introduction of rough path theory by Lyons \cite{LYO98}, and of the sewing lemma by Feyel–de La Pradelle \citep{FLP06} and Gubinelli \citep{GUB04}. 
These developments have motivated the generalization of Young’s integral to higher dimensions~$d.$

In \citep{AlSTTRE24}, in dimension $2$ and still under assumption \eqref{eq : alpha + beta > d} on the exponents $\alpha, \beta_1, \beta_2,$ the authors directly define the integral of $f \dg^1 \wedge \dg^2$ over any oriented triangle $[pqr] \subset \R^2$,
$$
\int_{[pqr]} f \dg^1 \wedge \dg^2,
$$
without first constructing the one-dimensional integral. Their construction relies on the notion of a \emph{germ}, namely a real-valued function defined on the set of triangles, playing the role of a discrete approximation of the integral. The integral is then defined as the limit of sums of germs evaluated over increasingly fine triangular subdivisions.  
Dimension $2$ plays a crucial role here, and it is not clear how to adapt their construction to oriented simplices of arbitrary dimension. Nevertheless, their definition of the integral in dimension $2$ is particularly robust: they show that several natural germs lead to the same integral, and that the way triangles are subdivided has little influence.  
Then they extend their integral by additivity to any oriented simple polygon, and more generally to any finite disjoint union of such polygons.
Finally, they establish a new condition on the upper box dimension of the boundary of a set $\Omega \subset \R^2$, namely
$$
\dimb(\partial \Omega) < \beta_1 + \beta_2,
$$
see Definition~\ref{defi : dim boite sup},
which allows one to define the integral
$$
\int_{\Omega} f \dg^1 \wedge \dg^2,
$$
from the integrals over the dyadic cubes intersecting $\Omega.$  
Stepanov and Trevisan also develop the notion of germs in great detail in \cite{STTR21}. In particular, they propose a generalization of the sewing lemma in dimension~$2$ and apply it to recover both Young’s integral and Züst’s integral.

In \citep{BOU24}, Bouafia also investigates the problem of generalizing Young’s integral to higher dimensions. He works on the cube $[0,1]^d$ and employs the concept of a \emph{charge}, which has already been successfully used in geometric measure theory to give meaning to certain integrals.  
A charge is a map that assigns a real number to each set of finite perimeter, is finitely additive, and satisfies a certain continuity property.  
Despite their apparent similarity, a charge is a different object from a measure. However, any measure with a density with respect to the Lebesgue measure is a charge, and in dimension $1$, it can be shown that every charge is the distributional derivative of a continuous function.  
Bouafia further develops this concept by introducing the notion of \emph{Hölder charge}, which quantifies the regularity of a charge, and studies the decomposition of such objects over a $d$-dimensional generalization of the Faber–Schauder basis.
He then uses this decomposition to construct the integral
$$
\int_{[0,1]^d} f \, \omega,
$$
when $f$ is a Hölder function and $\omega$ is a Hölder charge whose Hölder exponents sum to a sufficiently large value.  
He further associates to such a pair $(f,\omega)$ a new Hölder charge $f \cdot \omega$, which allows to extend the definition of his integral to any set of finite perimeter.  
Under assumption \eqref{eq : alpha + beta > d}, he shows that one can naturally associate to the functions $g^1, \dots, g^d$ a Hölder charge $\dg^1 \wedge \dots \wedge \dg^d$, which coincides with the measure $\det(\dg(x)) \, \d x$ when the functions $g^1, \dots, g^d$ are Lipschitz.  
This construction extends Züst’s integral to bounded sets of finite perimeter in arbitrary dimension.
This work is closely related to \cite{BouDeP25}, where De~Pauw and the previous author pursue the theoretical study of charges and apply their results to fractional Brownian sheets.

In \citep{ChSi24}, the authors consider a more general geometric setting, that of a $d$-dimensional manifold.  
For each integer $1 \le k \le d$, their goal is to give a rigorous meaning to the integral of a \emph{distributional $k$-form} over a smooth submanifold of dimension~$k$.  
Their work is motivated by the study of random fields in probability theory, where one seeks to make sense of the integration of random distributions defined on~$\R^d$ over $k$-dimensional geometric objects when $k < d$.  
Examples include the study of \emph{circle averages of the Gaussian Free Field} and that of \emph{Wilson loop observables} in gauge theory. Their approach is based on geometric integration techniques using simplices and following the seminal ideas introduced by Whitney and Harrison.

\medskip

Our aim in this paper is to give a meaning to the integral
\[
\int_{\Omega} \fdgd
\]
within a classical distributional framework, under more general geometric conditions on the domain~$\Omega$.  
Our approach takes place in a traditional analytical setting: we provide a natural and explicit construction of the distribution $\fdgd$ in the sense of~Schwartz on~$\R^d$, study its Hölder regularity through its wavelet coefficients, and define its integral by duality, via the standard correspondence between dual Besov spaces. We further connect the analytical condition to measure geometric properties of the domain.  

\subsection{Our contribution}
Let $\alpha, \beta_1, \dots, \beta_d$ be real numbers satisfying \eqref{eq : alpha + beta > d}. Set
\begin{equation}
\label{eq : def de beta et gamma}
\beta = \sum_{i=1}^d \beta_i \quad \text{and} \quad  \gamma = d - \beta.
\end{equation}
The definitions of globally Hölder function and distribution spaces on $\R^d$ are recalled in Section \ref{section : Notations globales} (see in particular \eqref{def : Espace d'Hölder global} and \eqref{eq : norme Cgamma}). The main result of Section \ref{section : Construction et étude de la distribution : cas globalement höldérien} is the construction, from Hölder functions $f, g^1, \dots, g^d,$ a Hölder distribution $\fdgd$ and to clarify how this distribution depends on these functions.

\begin{theo}
\label{theo : distribution globale f dg1 ^....^ dgd}
Under the above assumptions, there exists a unique $(d+1)$-linear mapping 
$$\begin{array}{rrcl}
\M :  &\Calpha(\R^d) \times \C^{\beta_1}(\R^d) \times \dots \times \C^{\beta_d}(\R^d) & \longrightarrow & \C^{-\gamma}(\R^d)  \\
& (f, g^1, \dots, g^d) & \longmapsto & \fdgd
\end{array}$$
satisfying the following properties:
\begin{enumerate}[label=\arabic*.]
\item If the functions $g^1, \dots, g^d$ are Lipschitz, then 
\begin{equation}
\label{eq : prop 1 theo distribution globale f dg1 ^....^ dgd}
\fdgd = f \det(\dg).
\end{equation}
\item Let $\alpha' \le \alpha, \, \beta_1' \le \beta_1, \, \dots \, , \, \beta_d' \le \beta_d$ satisfy \eqref{eq : alpha + beta > d}, and set $\gamma' = d - \sum_{i=1}^d \beta_i'$.  
Then, there exists a constant $C > 0$ such that, for all $f \in \C^{\alpha}(\R^d), g^1 \in \C^{\beta_1}(\R^d), \dots, g^d \in \C^{\beta_d}(\R^d)$, we have
\begin{equation}
\label{eq : prop 2 theo distribution globale f dg1 ^....^ dgd}
\|\fdgd\|_{\C^{-\gamma'}} \le C\, \|f\|_{\C^{\alpha'}}\prod_{i=1}^d\|g^i\|_{\beta_i'}.
\end{equation}
\end{enumerate}
\end{theo}

To prove the uniqueness of such a mapping $\M,$ we shall use the fact that, by virtue of Lemma $\ref{lem : lemme d'approximation global},$ any function in the space $\Calpha(\R^d)$ can be approximated in the norm $\C^{\alpha'}$ by a sequence of arbitrarily smooth functions, whenever $\alpha' \in (0,\alpha)$. 
This approximation result also allows us to deduce from the smooth case several properties satisfied by the mapping $\M$:
\begin{enumerate}[label=\arabic*.]
\item the mapping $\M$ is antisymmetric with respect to the functions $(g^i)_{i=1}^d,$
\item if one of the functions $g^i$ is constant, then 
$ \fdgd = 0.$
\end{enumerate}

\medskip

Let us now fix
$$ (f, g^1, \dots, g^d) \in \Calpha(\R^d) \times  \C^{\beta_1}(\R^d)\times \dots \times \C^{\beta_d}(\R^d).$$
In Section \ref{section : Extension de l'intégrale de Züst à un domaine borné général} we use the distribution $\fdgd$ to give  meaning to integrals of the form 
$\int_{\Omega} \fdgd$, for general Borel sets $\Omega$ in $\R^d$.
Since the space $\C^{-\gamma}(\R^d)$ is the topological dual of the Besov space $\B^{\gamma}_{1,1}(\R^d)$ (see definition \ref{def : espace de Besov globaux}), we define by duality the integral by setting
 $$\int_{\Omega} \fdgd = \langle \fdgd, \mathbb{1}_{\Omega} \rangle$$
for every Borel set $\Omega \in \B(\R^d)$ satisfying the condition 
\begin{equation}
\label{eq : condition besov pour l'intégral}
\mathbb{1}_{\Omega} \in \B^{\gamma}_{1,1}(\R^d);
\end{equation}
note that this condition does not imply that $\Omega$ is a bounded set.
We shall verify that this integral coincides with the one introduced by Züst in Theorem \ref{theo : intégrale de Zust} when the set $\Omega$ is a rectangle in $\R^d,$ and that this condition generalizes the condition on the perimeter of $\Omega$ used by Bouafia in \citep{BOU24}.

Subsequently, we will provide a simple geometric criterion on the Borel set $\Omega$ ensuring that condition \eqref{eq : condition besov pour l'intégral} is satisfied. More precisely, we introduce the notion of  \emph{Lebesgue boundary} of the set $\Omega$ (a subset of the topological boundary, see Definition \ref{def : frontière de Lebesgue}), which allows to establish a geometric criterion  which guarantees that the function $\mathbb{1}_{\Omega}$ belongs to a given Besov space, see Proposition \ref{prop : Espace de Besov indicatrice en dim d}. This criterion refines the classical one involving the upper box dimension of the usual boundary (see Definition \ref{defi : dim boite sup}), and relies crucially on the fact that it is always possible to construct a Borel set $\widetilde\Omega$ such that 
$$\lambda \left( \Omega \triangle \widetilde\Omega \right) = 0  \quad \text{ and } \quad \partial^{\ast} \Omega = \partial \widetilde\Omega,$$
see Theorem \ref{theo : frontière de Lebesgue}. 
For any subset $A \subset \R^d$ and any integer $j \in \N,$ we will use the notation  
\begin{equation}
\label{eq : def de Nj(A)}
 N_j(A)  = \Card\Big( \big\lbrace (k_1, \dots, k_d) \in \Z^d : A \cap \prod_{i=1}^d\left[ \tfrac{k_i}{2^j}, \tfrac{k_i+1}{2^j}\right] \neq \varnothing \big\rbrace\Big).
\end{equation}
Proposition \ref{prop : Espace de Besov indicatrice en dim d} then yields the following criterion:
\begin{equation}
\label{eq : indicatrice appartient à B(gamma,1,1)}
\sum_{j=0}^{\infty} 2^{- \beta j}N_j(\partial^{\ast} \Omega) < \infty \implies \mathbb{1}_{\Omega} \in \B^{\gamma}_{1,1}(\R^d).
\end{equation}
We will observe in Corollary \ref{cor : dim boîte indicatrice appartient à B(gamma,1,1)} that a condition on the upper box dimension of the Lebesgue boundary of the set $\Omega$ implies the geometric convergence of the above series. This fact, together with the inclusion $\partial^{\ast} \Omega  \subset \partial \Omega$, allows to recover the condition of Alberti, Stepanov, and Trevisan in dimension $2,$ as stated in \citep{AlSTTRE24}. We construct a class of functions, defined as Schauder series, such that their epigraphs satisfy \eqref{eq : indicatrice appartient à B(gamma,1,1)} in a sharp way.

\medskip

In Section \ref{section : Towards more general assumptions}, we relax the Hölder regularity assumptions on the functions $f,g^1, \dots, g^d,$ so that it remains possible to define the Hölder distribution $\fdgd,$ and to use it to define the integral 
$$\int_{\Omega} \fdgd$$
for any Borel set $\Omega$ satisfying condition \eqref{eq : condition besov pour l'intégral}.

First, we highlight the fact that the function $f$  on the one hand, and the functions $g^1, \dots, g^d$ on the other hand,  play distinct roles in the above construction of this distribution. More precisely, we show that if the function $f$ belongs to the inhomogeneous Hölder space $\Calpha(\R^d),$ it is sufficient to assume that the functions $g^1, \dots, g^d$ belong respectively to the homogeneous Hölder spaces $\dot \C^{\beta_1}(\R^d), \dots, \dot \C^{\beta_d}(\R^d),$ in order to define the distribution $\fdgd$ through \eqref{eq : def de la distribution fdg1^...^dgd}. We refer to the beginning of Section \ref{section : espace de hölder homogène} for the precise definition of these spaces. We then characterize this distribution in Theorem \ref{theo : distribution globale f dg1 ^....^ dgd homogène}, which provides a homogeneous version of Theorem \ref{theo : distribution globale f dg1 ^....^ dgd}. Finally, we note that in this setting it is still possible to define the integral $\int_{\Omega} \fdgd$ via \eqref{eq : Int_Omega fdgd = < fdgd, 1_Omega> }, for any Borel set $\Omega$ satisfying condition \eqref{eq : condition besov pour l'intégral}.

Secondly, we consider the case where the functions $f,g^1, \dots, g^d$ are only locally Hölder continuous.  
In this case, it is still possible to define the distribution $\fdgd$ via \eqref{eq : def de la distribution fdg1^...^dgd}, but this distribution now belongs to a space of locally Hölder distributions. We refer to the beginning of Section \ref{section : Construction et étude de la distribution : cas localement höldérien} for the precise definition of these spaces.  
We then characterize this distribution by means of a local version of Theorem \ref{theo : distribution globale f dg1 ^....^ dgd}.  

\begin{theo}
\label{theo : distribution locale f dg1 ^....^ dgd}
There exists a unique $(d+1)$-linear map
$$\begin{array}{rrcl}
\M :  &\Calphaloc(\R^d) \times  \C^{\beta_1}_{\loc}(\R^d) \times \dots \times \C^{\beta_d}_{\loc}(\R^d) & \longrightarrow & \C^{-\gamma}_{\loc}(\R^d)  \\
& (f, g^1, \dots, g^d) & \longmapsto &  \fdgd
\end{array}$$
satisfying the following properties:
\begin{enumerate}[label=\arabic*.]
\item If the functions $g^1, \dots, g^d$ are locally Lipschitz, then
$$ \fdgd = f\det(\dg).$$
\item Let $\alpha' \le \alpha$ and $\beta_i'\le \beta_i,$ $i \in \{1,\dots,d\},$ satisfying \eqref{eq : alpha + beta > d}.  
Set $\gamma' = d-\sum_{i=1}^d \beta_i'.$  
Then, there exists a constant $C > 0$ such that, for every $f\in \Calphaloc(\R^d), g^1 \in \C^{\beta_1}_{\loc}(\R^d), \dots, g^d \in \C^{\beta_d}_{\loc}(\R^d)$ and every integer $n \in \N,$ one has
$$ \|\fdgd\|_{\C^{-\gamma'}(\nn)} \le C\, \|f\|_{\C^{\alpha'}([-n,n+N]^d)} \prod_{i=1}^d \|g_i\|_{\beta_i', [-n,n+N]^d}.$$
\end{enumerate}
\end{theo}
Finally, we show that the integral  
\[
\int_{\Omega} \fdgd
\]
can still be defined by means of a localization argument, for any bounded Borel set $\Omega$ satisfying condition~\eqref{eq : condition besov pour l'intégral} (see~\eqref{eq : Int_Omega fdgd = < fdgd, 1_Omega> cas local}).

\section{Construction and study of the distribution: the globally Hölder case}
\label{section : Construction et étude de la distribution : cas globalement höldérien}
\subsection{Notations and background}
\label{section : Notations globales}

Let $d \in \N^{\ast}.$ In this text, we will consider real-valued functions or distributions defined on the whole space $\R^d$. Let us begin by recalling some definitions and notations used in what follows.

Let $f : \R^d \longrightarrow \R$ be a function. For any multi-index $k = (k_1, \dots, k_d) \in \N^d$, the partial derivative of order $k$ of the function $f$ at a point $x \in \R^d$ is denoted by
$$
\partial^{k} f (x) = \frac{\partial^{|k|} f}{\partial x_1^{k_1} \dots \partial x_d^{k_d}}(x),
$$
where $|k| = \sum_{i=1}^d k_i$ is the length of the multi-index $k$. We adopt the convention  $\partial^{k}f = f$ if the multi-index $k$ is zero.
For any integer $r \in \N$, the function $f$ is of class $\C^r$ on $\R^d$ if all partial derivatives $\partial^{k}f$, for any $k \in \N^d$ with $|k| \le r$, exist and are continuous in $\R^d$.
We denote by $\C^r(\R^d)$ the space of functions of class $\C^r$ on $\R^d$ that are bounded on $\R^d$, as well as all their partial derivatives of order less than or equal to $r$. We will denote
$$ \|f\|_{\infty} = \sup_{x \in \R^d} |f(x)|, $$
and endow the space $\C^r(\R^d)$ with the norm
$$ \|f\|_{\C^r} = \sum_{\substack{ k \in \N^d \\ |k|\le r}} \|\partial^{k}f\|_{\infty}.$$
The function $f$ is said to be infinitely differentiable on $\R^d$, or of class $\C^{\infty}$, if it is of class $\C^r$ on $\R^d$ for all integers $r \in \N$. The space of functions of class $\C^{\infty}$ on $\R^d$ whose derivatives, as well as the function itself, are bounded on $\R^d$ is denoted by $\C^{\infty}(\R^d).$
A test function is any function of class $\C^{\infty}$ with compact support in $\R^d.$ We write $\D(\R^d)$ for the space consisting of such functions. For any compact set $K \subset \R^d$, let $\D(K) = \lbrace \phi \in \D(\R^d) : \supp(\phi) \subset K \rbrace$.
The space of distributions on $\R^d$ will be denoted by $\D'(\R^d)$. It consists of linear forms $T : \D(\R^d) \longrightarrow \R$ satisfying the following continuity property: for any compact set $K \subset \R^d$, there exist a constant $C = C_K > 0$ and an integer $r = r_K \in \N$ such that, for any test function $\phi \in \D(K)$,
$$ |T(\phi)| \le C\, \|\phi\|_{\C^r}.$$
If there exists an integer $r$ that does not depend on the compact set $K$ in the above inequality, we say that the distribution $T$ is of finite order, and we define the order of $T$ as the smallest such integer $r$.
We will sometimes denote $\langle T, \phi \rangle$ the quantity $T(\phi)$. We now recall the definition of globally Hölder functions and distributions on $\R^d$.

Let $\alpha \in (0,1)$. We define the $\alpha$-Hölder seminorm of the function $f$ by  
$$ \|f\|_{\alpha} = \sup_{\substack{x, y \in \R^d \\ x \neq y}} \frac{|f(y) - f(x)|}{|y - x|^{\alpha}},$$  
and we say that $f$ is $\alpha$-Hölder (continuous) on $\R^d$ if this quantity is finite.  
We denote by  
\begin{equation}
\label{def : Espace d'Hölder global}
\Calpha(\R^d) = \left\{ f : \R^d \to \R \;:\; \|f\|_{\alpha} + \|f\|_{\infty} < \infty \right\}
\end{equation}  
the inhomogeneous Hölder space of bounded $\alpha$-Hölder functions on $\R^d$, endowed with the norm  
$$ \|f\|_{\Calpha} = \|f\|_{\alpha} + \|f\|_{\infty}. $$  
Recall that for every $\beta \in (0,\alpha)$ we have the continuous embeddings  
$$ \C^{1}(\R^d) \hookrightarrow \Calpha(\R^d) \hookrightarrow \Cbeta(\R^d) \hookrightarrow \C^{0}(\R^d). $$  

We now introduce the (global) Besov spaces $\Bspq(\R^d),$ where $s \in \R$ and $p,q \in (0,\infty].$
There are several ways to define these spaces, and we refer the reader to Triebel’s monographs \citep{Tri1, Tri3} for a comprehensive presentation. 
Here we choose to define them through their wavelet characterization, following Meyer’s book \citep{ME92}. 
We begin by recalling some results and notations concerning wavelets. For more details on this topic, we refer to the standard references \citep{Dau92}, \citep{Ma09}, and \citep{ME92}. 
The following theorem, due to Daubechies, states the existence of wavelet bases with compact support and arbitrarily high (but finite) smoothness (see Section $4.C$, p.$978$ of \citep{Dau88}).

\begin{theo}
\label{theo : Base d'ondelettes de Daubechies}
For any integer $r \ge 1$, there exist real-valued functions 
$\phi, \psi^{(1)}, \dots, \psi^{(2^d-1)}$ defined on $\R^d$ and an integer $N \in \N$ 
satisfying the following properties:
\begin{enumerate}[label=\arabic*.]
\item The functions $\phi, \psi^{(1)}, \dots, \psi^{(2^d-1)}$ are of class $\C^r$ with compact support, and all have the same support exactly equal to the cube $[0,N]^d.$
\item For every multi-index $m \in \N^d$ such that $|m| \le r - 1$ and every integer $i \in \{1, \dots, 2^d-1\}$, one has
\[
\int_{\R^d} \phi(x) \dx = 1,
\qquad 
\int_{\R^d} x^m \psi^{(i)}(x) \dx = 0.
\]
\item The family 
\[
\{\phi(\cdot - k),\; 2^{\frac{dj}{2}} \psi^{(i)}(2^j \cdot - k) 
: j \in \N,\; i \in \{1, \dots, 2^d - 1\},\; k \in \Z^d \}
\]
forms an orthonormal basis of $L^2(\R^d).$
\end{enumerate}
\end{theo}

In the following, we fix an integer $r \in \N$ large enough so that the wavelet characterizations of the functional spaces considered are valid, and wavelets  $\phi, \psi^{(1)}, \dots, \psi^{(2^d-1)}$ of class $\C^r$ satisfying the properties of the Theorem \ref{theo :  Base d'ondelettes de Daubechies}. 
For every integer $i \in \{1, \dots, 2^d - 1\}$, every integer $j \in \N$, and every multi-index $k = (k_1, \dots, k_d) \in \Z^d$, we use the notations

\begin{equation}
\label{eq : notation phik psiijk}
\phi_k = \phi(\cdot-k) \quad \text{and} \quad \psiijk = \psi^{(i)}(2^j \cdot-k).
\end{equation}
We then have the following equalities for the supports of these wavelets:
\[
\supp(\phi_k) = \prod_{l=1}^d [k_l, k_l + N]
\quad \text{and} \quad  
\supp\bigl(\psiijk\bigr) = \prod_{l=1}^d \left[ \tfrac{k_l}{2^j}, \tfrac{k_l+N}{2^j}\right],
\]
and, for simplicity, we shall write
\[
[k, k+N] = \prod_{l=1}^d [k_l, k_l + N]
\quad \text{and} \quad 
\kNj = \prod_{l=1}^d \left[ \tfrac{k_l}{2^j}, \tfrac{k_l+N}{2^j}\right].
\]
We now recall the definition of wavelet coefficients of a distribution of order less than or equal to $r$ on $\R^d$ on this wavelet basis. 
\begin{defi}
\label{def : coeff d'ondelettes dim d}
Let $f \in \D'(\R^d)$ be a distribution of order less than or equal to $r$. For all integers $j \in \N$, $i \in \{1, \dots, 2^d-1\}$ and all multi-indices $k \in \Z^d$, we define the wavelet coefficients of $f$ by
$$c_k = \ck(f) = \langle f , \phi_k \rangle \qquad \text{ and } \qquad \cijk = \cijk(f) = 2^{dj} \langle f, \psiijk \rangle.$$
\end{defi}
Although the wavelet coefficients of a distribution depend on the chosen wavelet basis, we will see that in practice the choice of basis has little impact; for this reason, we shall not indicate the dependence of these coefficients on the basis.
We are now in a position to state the definition of the global Besov spaces on $\R^d$ that we will use.
\begin{defi}
\label{def : espace de Besov globaux}
Let $s \in \R,$ and $p, q \in (0,\infty].$ A distribution $f \in \D'(\R^d)$ of order less than or equal to $r$ belongs to the space $\Bspq(\R^d)$ if its wavelets coefficients satisfy
$$\sum_{k\in \Z^d} |c_k|^p < \infty \quad  \text{and} \qquad 
\sum_{j=0}^{\infty} \left(\sum_{k\in \Z^d} \sum_{i=1}^{2^d-1} 
\left| 2^{(s-\frac{d}{p})j}\cijk\right|^p \right)^{\frac{q}{p}} <\infty,$$
with the usual modifications if $p$ or $q$ is infinite.
\end{defi}

We mention that this definition does not depend on the particular (sufficiently regular) wavelet basis chosen, and that the quantities
\begin{equation}
\label{eq : normes Bspq}
\left( \sum_{k\in \Z^d} |c_k|^p \right)^{\frac{1}{p}} +  
\left( \sum_{j=0}^{\infty} \left( 
\sum_{k\in \Z^d} \sum_{i=1}^{2^d-1} \left| 2^{(s-\frac{d}{p})j}\cijk\right|^p \right)^{\frac{q}{p}} \right)^{\frac{1}{q}},
\end{equation}
while depending on the chosen wavelet basis, define equivalent norms whenever $p,q \in [1,\infty]$. 
For further details, we refer the reader to \citep{ME92}. 
In what follows, we shall denote by $\|f\|_{\Bspq(\R^d)}$ any of the norms defined in \eqref{eq : normes Bspq}.

Let $\alpha \in (0,1)$. We recall that the space $\Calpha(\R^d)$ coincides with the space $\B^{\alpha}_{\infty, \infty}(\R^d)$ and the corresponding norms are equivalent (see Chapter 6 of \citep{ME92}). This motivates the following definition.

Let $\gamma \in \R \setminus \N$.  
A distribution $f$ is said to have global $\gamma$–Hölder regularity if $f \in \B^{\gamma}_{\infty, \infty}(\R^d)$.  
We denote this space by $\Cgamma(\R^d)$, and recall that the norm $\|f\|_{\Cgamma(\R^d)}$ is the infimum of the constants $C > 0$ such that, for all $i \in \{1, \dots, 2^d - 1\}$, $j \in \N$, and $k \in \Z^d$,
\begin{equation}
\label{eq : norme Cgamma}
|c_k| \le C, 
\qquad \text{and} \qquad 
\big|\cijk\big| \le C\, 2^{-\gamma j}.
\end{equation}

\subsection{Construction of the distribution}
\label{section : construction de la distribution globale}
In what follows, we fix real numbers $\alpha, \beta_1, \dots, \beta_d$ satisfying \eqref{eq : alpha + beta > d}, and denote by $\beta$ and $\gamma$ the reals defined by \eqref{eq : def de beta et gamma}.
The aim of this section is to prove Theorem  \ref{theo : distribution globale f dg1 ^....^ dgd}. To this end, we begin by constructing, for all Hölder functions
$$ f \in \Calpha(\R^d), g^1 \in \C^{\beta_1}(\R^d), \dots, g^d \in \C^{\beta_d}(\R^d), $$
a distribution denoted by $\fdgd$, and satisfying
\begin{equation}
\label{eq :  f dg1 ^...^dgd = fdet(dg)}
\fdgd = f \det(\dg)
\end{equation}
when the functions $g^i$ are Lipschitz.
We then determine the global Hölder regularity of this distribution and show that:
$$ \fdgd \in \C^{-\gamma}(\R^d). $$
At this stage, we will be able to define the mapping $\M$ of Theorem  \ref{theo : distribution globale f dg1 ^....^ dgd}:
$$\begin{array}{rrcl}
\M : & \Calpha(\R^d) \times \C^{\beta_1}(\R^d) \times \dots \times \C^{\beta_d}(\R^d) & \longrightarrow & \C^{-\gamma}(\R^d) \\
& (f, g^1, \dots, g^d) & \longmapsto & \fdgd.
\end{array}$$
We will then study the properties of this mapping and complete the proof of the theorem.

Let us emphasise right away that neither the space $\C^{\infty}(\R^d),$ nor even the space of bounded Lipschitz functions on $\R^d,$ is dense in $\C^{\alpha}(\R^d)$. There is therefore no a priori reason to expect the existence of a unique mapping $\M$ as above. However, $\C^{\infty}(\R^d)$ is dense in the space $\left(\Calpha(\R^d), \|\cdot\|_{\C^{\alpha'}}\right)$ for any $\alpha' \in\, (0,\alpha)$. This fact, and a continuity property of the mapping $\M$ stronger than usual continuity, which we will establish, will allow us to show that $\M$ is the only mapping which verifies this continuity property and identity \eqref{eq :  f dg1 ^...^dgd = fdet(dg)}.

\medskip

Let $f \in \Calpha(\R^d)$ and $g^1 \in \C^{\beta_1}(\R^d), \dots, g^d \in \C^{\beta_d}(\R^d)$. 
For any rectangle $R = [a_1, b_1] \times \dots \times [a_d, b_d]$ in $\R^d$, we denote by 
$\delta(R)$ the quantity $\max_{1 \le i \le d} (b_i - a_i)$. 
We begin by recalling the sewing inequality \eqref{eq : inégalité de couture} in this setting:
\begin{equation}
\label{eq : borne de couture dans le cas globalement hölder}
\left| \int_{R} \fdgd - f(a_1, \dots, a_d) \int_{R} \dgd \right|
\le C\, \delta(R)^{\alpha' + \beta'}\, \|f\|_{\alpha',R} 
\prod_{i=1}^d \|g^i\|_{\beta_i',R},
\end{equation}
where we recall that $\beta_1' \in (0, \beta_1], \dots, \beta_d' \in (0, \beta_d]$ 
satisfy condition~\eqref{eq : alpha + beta > d}, that $\beta' = \sum_{i=1}^d \beta_i'$, 
and that the constant $C$ depends only on $d, \alpha', \beta_1', \dots, \beta_d'$.
Moreover, the Hölder seminorms on $R$ are bounded by those on $\R^d$.
From the additivity property on rectangles (Proposition~ \ref{prop : additivité sur les rectangles}), we deduce the following result.

\begin{lemme}
\label{lem : prolongement de l'intégrale quand la fonction vaut 0}
If the function $f$ has compact support, then for any rectangles $P$ and $Q$ containing the support of $f$, we have
$$ \int_P \fdgd = \int_Q \fdgd. $$
\end{lemme}

\begin{proof}
The set $P \cap Q$ is a rectangle that contains the support of $f$ so it sufficient to prove the equality when $P$ is strictly contained in $Q$. We decompose $Q$ as a finite union of disjoint interior rectangles:
$$ Q = P \cup \bigcup_{j=1}^m P_j. $$
For each $j \in \{1,\dots,m\}$, the restriction of $f$ to $P_j$ is zero, so by linearity of the integral, $\int_{P_j} \fdgd = 0$. Using Proposition~ \ref{prop : additivité sur les rectangles}, we obtain
$$ \int_Q \fdgd = \int_P \fdgd + \sum_{j=1}^m \int_{P_j} \fdgd = \int_P \fdgd, $$
which proves the result.
\end{proof}

We now have all the tools needed to define our distribution.
Züst’s integral allows to define this distribution as a continuous linear form on a space of functions larger than the space of test functions. Denote by
$$ \Calpha_c(\R^d) = \left\{ h \in \C^{\alpha}(\R^d) : \supp(h) \text{ is compact} \right\} $$
the space consisting of real-valued $\alpha$-Hölder functions on $\R^d$ with compact support. 
For every integer $n \in \N,$ we shall use the notations $\| \cdot\|_{\alpha, [-n, n]^d}$ and $\| \cdot\|_{\Calpha([-n,n]^d)}$ introduced at the beginning of Section \ref{section : Construction et étude de la distribution : cas localement höldérien}.
Recall that the space $\Calpha_c(\R^d)$ endowed with the family $\{ \|\cdot\|_{\Calpha(\nn)} : n \in \N \}$ is a Fréchet space.
Since $\Calpha(\R^d)$ is a Banach algebra, Lemma  \ref{lem : prolongement de l'intégrale quand la fonction vaut 0} allows us to define the mapping
$$\begin{array}{rlcl}
\fdgd : & \Calpha_c(\R^d) & \longrightarrow & \R \\
& h & \longmapsto & \int_{\R^d} (fh) \dgd
\end{array}$$
where this integral can in fact be taken over any rectangle containing the support of $h$.

\begin{prop}
\label{prop : fdg forme linéaire sur Calpha}
The mapping $\fdgd$ defines a continuous linear form on the space $\Calpha_c(\R^d)$.
\end{prop}

\begin{proof}
The linearity of this mapping follows from the linearity of Züst’s integral (see Theorem  \ref{theo : intégrale de Zust}) with respect to the first function, when the rectangle over which we integrate is fixed. We now prove continuity.
Let $n \in \N.$ Let $h \in \Calpha_c(\R^d)$ be a function whose support is contained in the cube $\nn.$
By definition, we have
$$ \langle \fdgd , h \rangle = \int_{\nn} (fh) \dgd.$$
Since the function $h$ vanishes on the boundary of the cube $\nn,$ we deduce from the sewing inequality $\eqref{eq : borne de couture dans le cas globalement hölder}$ the estimate
$$ \left| \int_{\nn} (fh) \dgd \right| \le C_n \|fh\|_{\alpha, \nn} \, ,$$
with $C_n = C\, n^{\alpha + \beta}  \prod_{i=1}^d \|g^i\|_{\beta_i}.$
The conclusion follows from the inequality
$$ \|fh\|_{\alpha,\nn} \le \|f\|_{\Calpha} \|h\|_{\Calpha(\nn)},$$
which is straightforward.
\end{proof}

Since $\D(\R^d)$ is included in $\Calpha_{c}(\R^d)$, the inequality $\|\cdot\|_{\Calpha(\R^d)} \le 3 \|\cdot\|_{\C^1(\R^d)}$ shows that the restriction of the mapping $\fdgd$ to $\D(\R^d)$ defines a distribution of order at most $1$ on $\R^d$. We will also denote by $\fdgd$ this distribution and recall that it is defined for any test function $\phi \in \D(\R^d),$ by
\begin{equation}
\label{eq : def de la distribution fdg1^...^dgd}
\langle \fdgd, \phi \rangle = \int_{\R^d} (f\phi) \dgd,
\end{equation}
where the integral is taken over any rectangle containing the support of $\phi$.
We now show that our distribution satisfies the desired identity \eqref{eq :  f dg1 ^...^dgd = fdet(dg)}.

\begin{lemme}
\label{lem : f dg1 ^...^dgd = fdet(dg) globalement hölder}
Suppose that the functions $g^1, \dots, g^d$ are Lipschitz and bounded on $\R^d$.
Then the distribution $\fdgd$ identifies with the essentially bounded function $f \det\left(\dg\right)$, in the sense that for every test function $\phi \in \D(\R^d)$, one has
$$ \langle \fdgd, \phi \rangle = \int_{\R^d} f(x) \det(\dg(x)) \phi(x) \dx. $$
\end{lemme}

\begin{proof}
For all $i \in \{1,\dots,d\}$, the function $g^i$ is Lipschitz and bounded on $\R^d$, so it belongs to the space $\C^{\beta_i}(\R^d)$ and the distribution $\fdgd$ is well-defined. Moreover, a theorem of Lebesgue ensures that the function $g$ is differentiable almost everywhere and its Jacobian is bounded on $\R^d$. In particular, we deduce that the function $f\det\left(\dg\right)$ is bounded on $\R^d$.

Let $\phi \in \D(\R^d)$ and $R$ be a rectangle in $\R^d$ containing the support of $\phi$. By definition, we have
$$ \langle \fdgd , \phi \rangle = \int_{R} (f\phi) \dgd. $$
The second property of Züst’s theorem (Theorem~ \ref{theo : intégrale de Zust}) gives
$$ \int_{R} (f\phi) \dgd = \int_{R} f(x)\det(\dg(x)) \phi(x) \dx $$
and we conclude using the fact that the support of $\phi$ is included in $R$.
\end{proof}

\subsection{Regularity of the distribution}
\label{section : régularité de la distribution globale}
Our goal now is to study the regularity of the distribution $\fdgd.$
We show that this distribution belongs to a globally Hölder-type distribution space defined in Section \ref{section : Notations globales} (Definition \ref{def : espace de Besov globaux}), and we provide an upper bound on its Hölder norm in terms of those of the functions $f, g^1, \dots, g^d.$
\begin{prop}
\label{prop : régularité de la distribution globalement höldérien}
Let $\alpha' \in (0, \alpha], \beta_1'\in (0, \beta_1], \dots, \beta_d'\in (0, \beta_d]$
be real numbers satisfying $\eqref{eq : alpha + beta > d}.$ Set $\gamma' =  d - \sum_{i=1}^d \beta_i'.$
Then, the distribution $\fdgd$ belongs to the space $\C^{-\gamma'}(\R^d).$
Moreover, 
there exists a constant $C > 0$ such that, for all functions $f\in \Calpha(\R^d), g^1 \in \C^{\beta_1}(\R^d)$, $\dots, g^d \in \C^{\beta_d}(\R^d),$ we have
$$ \|\fdgd\|_{\C^{-\gamma'}} \le C\,\|f\|_{\C^{\alpha'}}\prod_{i=1}^d \|g_i\|_{\beta_i'}.$$
\end{prop}

\begin{proof}
Set 
$$\beta' = \sum_{i=1}^{d} \beta_i' \quad \text{and} \quad K = N^{\alpha'+\beta'} \max_{i=1,\dots 2^d-1}\left( \|\phi\|_{\C^{\alpha'}(\R^d)}, \|\psi^{(i)}\|_{\C^{\alpha'}(\R^d)}\right).$$ 
Let $i \in \{1, \dots, 2^d -1\}, j \in \N \text{ and } k \in \Z^d.$ We shall now bound the wavelet coefficient $\cijk.$
The cube $\kNj$ contains the support of $\psiijk,$ hence we have the equality
$$\cijk = 2^{dj} \int_{\kNj} \left(f\psiijk\right) \dgd.$$
The sewing inequality \eqref{eq : borne de couture dans le cas globalement hölder}, the fact that the function $\psiijk$ vanishes on the boundary of the cube $\kNj,$ and the equality $\delta\big (\kNj\big) = \frac{N}{2^j}$ yield the inequality  
\begin{equation}
\label{eq : preuve régularité distribution global}
\big|\cijk\big| 
\le C\, N^{\alpha' + \beta'}\, 2^{(\gamma' - \alpha')j} 
\|f \psiijk\|_{\alpha'} 
\prod_{m=1}^d \|g^m\|_{\beta_m'}.
\end{equation}
where the constant $C > 0$ depends only on $d, \alpha',\beta_1',\dots, \beta_d'.$ 
From the inequality
$$ \|f \psiijk\|_{\alpha'} \le \|f\|_{\infty}\|\psiijk\|_{\alpha'} + \|\psiijk\|_{\infty}\|f\|_{\alpha'}$$
and the equalities 
$$ \|\psiijk\|_{\infty} = \|\psi^{(i)}\|_{\infty}, \qquad \|\psiijk\|_{\alpha'} = 2^{\alpha' j}\|\psi^{(i)}\|_{\alpha'}$$
we deduce the inequality 
$$\|f \psiijk\|_{\alpha'} \le 2^{\alpha' j} \|f\|_{\C^{\alpha'}} \|\psi^{(i)}\|_{\C^{\alpha'}}.$$
Inserting this inequality into \eqref{eq : preuve régularité distribution global}, we obtain
$$
\big|\cijk\big|
\le C\, K\, 2^{\gamma' j}
\|f\|_{\C^{\alpha'}} 
\prod_{m=1}^d \|g^m\|_{\beta_m'}.
$$ 
One proves in the same way the bound 
$$|c_k| \le C\, K\|f\|_{\C^{\alpha'}} \prod_{m=1}^d \|g^m\|_{\beta_i'}.$$
This implies
$$ \fdgd \in \C^{-\gamma'}(\R^d) \quad \text{ and } \quad \|\fdgd\|_{\C^{-\gamma'}} \le C\, K \|f\|_{\C^{\alpha'}}\prod_{i=1}^d \|g^i\|_{\beta_i'},$$
which completes the proof of the proposition.
\end{proof}

\subsection{Uniqueness and properties of the mapping $\M$}
\label{section : Unicité cas global}
In this section, we summarize the results obtained in the previous sections and complete the proof of Theorem $\ref{theo : distribution globale f dg1 ^....^ dgd}$.
Recall that for all functions 
$$ f\in \Calpha(\R^d), \, g^1 \in \C^{\beta_1}(\R^d), \dots,\, g^d \in \C^{\beta_d}(\R^d),$$ 
we have constructed a distribution $\fdgd$, defined in $\eqref{eq : def de la distribution fdg1^...^dgd}$, and shown in Proposition $\ref{prop : régularité de la distribution globalement höldérien}$ that this distribution belongs to the space $\C^{-\gamma}(\R^d).$
This implies that the mapping
 $$\begin{array}{rrcl}
\M :  &\Calpha(\R^d) \times  \C^{\beta_1}(\R^d) \times \dots \times \C^{\beta_d}(\R^d) & \longrightarrow & \C^{-\gamma}(\R^d)  \\
& (f, g^1, \dots, g^d) & \longmapsto &  \fdgd
\end{array}$$
is well defined.
This mapping is $(d+1)$-linear by virtue of the $(d+1)$-linearity of Züst’s integral on a fixed rectangle, recalled in Theorem $\ref{theo : intégrale de Zust}$.
On the one hand, we deduce from Lemma $\ref{lem : f dg1 ^...^dgd = fdet(dg) globalement hölder}$ that property \eqref{eq : prop 1 theo distribution globale f dg1 ^....^ dgd} of Theorem \ref{theo : distribution globale f dg1 ^....^ dgd} is satisfied.
On the other hand, Proposition \ref{prop : régularité de la distribution globalement höldérien} implies that the mapping $\M$ also satisfies property \eqref{eq : prop 2 theo distribution globale f dg1 ^....^ dgd} of Theorem \ref{theo : distribution globale f dg1 ^....^ dgd}.
It remains to prove the uniqueness of such a mapping.
To this end, we shall need the following approximation result.

\begin{lemme}[Global approximation lemma 1]
\label{lem : lemme d'approximation global}
Let $\alpha, \alpha' \in ( 0 ,1) $ with $\alpha' < \alpha.$ For every function $f\in \C^{\alpha}(\R^d),$ there exists a sequence of functions $(f_n)_{n\in\N} \in \C^{\infty}(\R^d)^{\N}$ such that
$$\|f_n\|_{\Calpha} \le \|f\|_{\Calpha} \quad \text{ and } \quad \underset{n \to \infty}{\lim} \|f-f_n\|_{\C^{\alpha'}} = 0.$$
\end{lemme}
We now consider an arbitrary $(d+1)$-linear mapping
$$\begin{array}{rrcl}
\B :  &\Calpha(\R^d) \times  \C^{\beta_1}(\R^d) \times \dots \times \C^{\beta_d}(\R^d) & \longrightarrow & \C^{-\gamma}(\R^d)  \\
& (f, g^1, \dots, g^d) & \longmapsto &  \B(f, g^1,\dots ,g^d)
\end{array}$$
satisfying properties \eqref{eq : prop 1 theo distribution globale f dg1 ^....^ dgd} and \eqref{eq : prop 2 theo distribution globale f dg1 ^....^ dgd} of Theorem \ref{theo : distribution globale f dg1 ^....^ dgd}.
For all functions $f\in \Calpha(\R^d)$, $ g^1 \in \C^{\beta_1}(\R^d)$, $\dots,\, g^d \in \C^{\beta_d}(\R^d),$ we shall prove that 
$$\B(f,g^1,\dots,g^d) = \M(f,g^1,\dots,g^d).$$
For each $i \in \{1, \dots, d\},$ let us choose a sequence of functions $(g^i_n)_{n\in\N}$ approximating the function $g^i$ in the sense of Lemma $\ref{lem : lemme d'approximation global}.$ 
The condition $\eqref{eq : alpha + beta > d}$ on the real numbers $\alpha, \beta_1, \dots, \beta_d$ implies that it is always possible to find real numbers 
$$\alpha' \in (0, \alpha), \beta_1'\in (0, \beta_1), \dots, \beta_d'\in (0, \beta_d) $$
satisfying the same condition. Set $\gamma ' = d - \sum_{i=1}^d \beta_i'.$  
For every $n \in \N,$ define 
$$\begin{array}{rlcl}
g_n :& \R^d & \longrightarrow & \R^d  \\
& x & \longmapsto & (g^1_n(x), \dots, g^d_n(x)) .
\end{array}.$$
On the one hand, each function $g^i_n$ is Lipschitz, hence property \eqref{eq : prop 1 theo distribution globale f dg1 ^....^ dgd} of Theorem $\ref{theo : distribution globale f dg1 ^....^ dgd}$ yields the equality 
$$ \M(f,g^1_n\dots,g^d_n) = f\det(\dg_n) = \B(f,g^1_n\dots,g^d_n).$$
On the other hand, each sequence of functions $(g^i_n)_{n\in\N}$ converges to the function $g^i$ in the space $\C^{\beta_i'}(\R^d).$ 
We therefore deduce from the $(d+1)$-linearity of the mappings $\M$ and $\B,$ and from property \eqref{eq : prop 2 theo distribution globale f dg1 ^....^ dgd} of Theorem $\ref{theo : distribution globale f dg1 ^....^ dgd},$ that the sequence of functions $(f\det(\d g_n))_{n\in\N}$ converges in the space $(\C^{-\gamma}(\R^d), \|\cdot\|_{\C^{-\gamma'}}) $ both to the distribution $\M(f,g^1,\dots,g^d)$ and to the distribution $\B(f,g^1,\dots,g^d),$ which yields the desired result.  

To conclude, we state several natural properties satisfied by the mapping $\M.$ These properties follow from the case where the functions $g^1,\dots, g^d$ are smooth and are proved using an approximation argument similar to the one we have just used. 
\begin{prop}
\label{prop : prop de l'appli M global}
The mapping $\M$ satisfies the following properties:
\begin{enumerate}[label=\arabic*.]
\item It is antisymmetric with respect to the functions $(g^i)_{i=1}^d.$
\item For all functions $f\in \C^{\alpha}(\R^d), g^1 \in \C^{\beta_1}(\R^d), \dots, g^d \in \C^{\beta_d}(\R^d),$ if one of the functions $g^i$ is constant, then 
$ \fdgd = 0.$
\end{enumerate}
\end{prop}

\section{Extension of Züst's integral to a general domain}
\label{section : Extension de l'intégrale de Züst à un domaine borné général}
Let $\alpha, \beta_1, \dots, \beta_d$ be real numbers satisfying \eqref{eq : alpha + beta > d}. Let $\beta$ and $\gamma$ be the reals defined by \eqref{eq : def de beta et gamma}.
In this section, we fix the functions
$$ f\in \Calpha(\R^d), \, g^1 \in \C^{\beta_1}(\R^d), \dots,\, g^d \in \C^{\beta_d}(\R^d), $$
and we define $\fdgd = \M(f,g^1,\dots,g^d)$ as the distribution obtained in Theorem  \ref{theo : distribution globale f dg1 ^....^ dgd}.
The goal of this section is to use this distribution to a give meaning to the integrals
$$ \int_{\Omega} \fdgd, $$
for certain Borel sets $\Omega \in \B(\R^d).$

\subsection{Definition of the integral}
We begin by recalling the following duality result between Besov spaces (Theorem $2.11.2$ of \citep{Tri1}).
\begin{theo}[Duality $\Bspq$]
For all $s \in \R, p \text{ and } q \in [1,\infty),$ 
$$ \left(\Bspq(\R^d)\right)' = \B^{-s}_{p',q'}(\R^d),$$
where $p'$ and $q'$ are the conjugate exponents of $p$ and $q.$
\end{theo}
For every Borel set $\Omega \in \B(\R^d)$ whose indicator function $\mathbb{1}_{\Omega}$ belongs to the space $\B^{\gamma}_{1,1}(\R^d),$ let
\begin{equation}
\label{eq : Int_Omega fdgd = < fdgd, 1_Omega> }
\int_{\Omega} \fdgd = \langle \fdgd, \mathbb{1}_{\Omega} \rangle.
\end{equation}
To justify this notation, we show that this integral coincides with Züst’s integral in the case where $\Omega$ is a rectangle.

\begin{prop}
Let $R =\prod_{i=1}^d [a_i, b_i]$ be a rectangle in $\R^d.$ Then, $\mathbb{1}_{R} \in \B^{\gamma}_{1,1}(\R^d)$ and 
$$\langle \fdgd , \mathbb{1}_{R} \rangle = \int_{R} \fdgd, $$
where the term on the right denotes the usual Züst integral defined in the setting of Theorem $\ref{theo : intégrale de Zust}$.
\end{prop}
To prove this proposition, and more precisely the fact that the indicator function of a rectangle $R$ belongs to the desired Besov space, we shall use a result established later (Proposition $\ref{cor : dim boîte indicatrice appartient à B(gamma,1,1)}$).

\begin{proof}
Recall that the upper box dimension (Definition $\ref{defi : dim boite sup}$) of the boundary of the rectangle $R$ is (less than or) equal to $d-1.$ Since $\beta > d-1,$ Proposition $\ref{cor : dim boîte indicatrice appartient à B(gamma,1,1)}$ implies that the function $\mathbb{1}_{R}$ indeed belongs to the space $\B^{\gamma}_{1,1}(\R^d).$

First assume that the functions $g^1, \dots, g^d$ are Lipschitz and bounded on $\R^d.$ The first property of Theorem $\ref{theo : intégrale de Zust}$ implies that Züst’s integral satisfies the following equality:
$$ \int_{R} \fdgd = \int_{R} f(x) \det(\dg(x)) \dx.$$ 
On the other hand, property \eqref{eq : prop 1 theo distribution globale f dg1 ^....^ dgd} of Theorem \ref{theo : distribution globale f dg1 ^....^ dgd} implies that the distribution $\fdgd$ identifies with the essentially bounded function $f\det(\dg).$ This yields the desired equality in the case where the functions $g^1, \dots, g^d$ are Lipschitz:
$$\langle \fdgd , \mathbb{1}_{R} \rangle = \int_{R} f(x) \det(\dg(x)) \dx = \int_{R} \fdgd.$$

Now assume that the functions $g^1, \dots, g^d$ are Hölder and bounded on $\R^d.$ For any integer $i \in \{1, \dots,d\},$ the approximation lemma $\ref{lem : lemme d'approximation global}$ allows us to approximate uniformly the function $g^i$ by a sequence of functions $(g^i_n)_{n\in \N}  \subset \C^{\infty}(\R^d)$ bounded in the space $C^{\beta_i}(\R^d).$ For fixed $n,$ we have shown that
\begin{equation}
\label{eq : preuve prop intégrale de zust coincide}
\langle f \dg^1_n \wedge \dots \wedge \dg^d_n , \mathbb{1}_{R} \rangle = \int_{R} f \dg^1_n \wedge \dots \wedge \dg^d_n.
\end{equation}
On the one hand, the continuity of Züst’s integral (Property $2$ of Theorem \ref{theo : intégrale de Zust}) implies that the right-hand side of this equality converges to the Züst integral $\int_{R} \fdgd.$ On the other hand,
condition $\eqref{eq : alpha + beta > d}$ implies that it is always possible to find real numbers 
$$ \beta_1'\in (0, \beta_1), \dots, \beta_d'\in (0, \beta_d) $$
so that the same condition is still satisfied with the indices $\alpha, \beta_1', \dots, \beta_d'.$ Set $\gamma ' = d - \sum_{i=1}^d \beta_i'.$ For each $i \in \{1, \dots, d \},$ the sequence of functions $(g^i_n)_{n\in\N}$ converges to the function $g^i$ in the space $\C^{\beta_i'}(\R^d),$ hence the continuity of the mapping $\M$ (Property \eqref{eq : prop 2 theo distribution globale f dg1 ^....^ dgd} of Theorem \ref{theo : distribution globale f dg1 ^....^ dgd}) implies that
the sequence of functions $(f \dg^1_n \wedge \dots \wedge \dg^d_n)_{n\in\N}$ converges in the space $(\C^{-\gamma}(\R^d), \|\cdot\|_{\C^{-\gamma'}})$ to the distribution $\fdgd.$ This implies that the left-hand side of equality \eqref{eq : preuve prop intégrale de zust coincide} converges to $\langle \fdgd, \mathbb{1}_{R} \rangle,$ and we deduce the desired result.
\end{proof}

Let us mention that Züst defines this integral in dimension $d$ when the domain of integration is a rectangle, or more generally, the image of a rectangle under a Lipschitz mapping (see \citep{Zu11}). Subsequently, Bouafia in \citep{BOU24} extended Züst’s integral to any bounded domain of finite perimeter. Condition \eqref{eq : condition besov pour l'intégral} is more general, which follows from Theorem 2 of \citep{Sic21}.  
Independently and using differents methods,  Alberti, Stepanov and Trevisan in \citep{AlSTTRE24} extended Züst’s integral, in dimension $2,$ to any bounded domain $\Omega \subset \R^2$ whose boundary has upper box dimension satisfying the condition
$$ \dimb(\partial \Omega) < \beta_1 + \beta_2.$$
Recall that in dimension $2,$ the boundary of a (non-degenerate) rectangle has dimension equal to $1,$ hence the boundary of its image under a (non-constant) Lipschitz mapping also has dimension $1.$ 
However, the real number $\beta_1 + \beta_2$ is strictly larger than $1,$ so the result of Alberti, Stepanov and Trevisan allows us to integrate $f \dg^1 \wedge \dg^2$ over sets with a fractal boundary of dimension larger than $1.$  
We shall show in Corollary \ref{cor : dim boîte indicatrice appartient à B(gamma,1,1)} that the geometric condition \eqref{eq : indicatrice appartient à B(gamma,1,1)} generalizes that of Alberti, Stepanov and Trevisan; thus, it concerns the Lebesgue boundary of the domain rather than the usual boundary (see Figure \ref{fig:objet_3}), and it is more precise than a condition involving only the upper box dimension (see the detailed example in Section \ref{section : graphe fonction hölder}).

\subsection{Regularity of an indicator function}
Let $\Omega\in \B(\R^d)$ be a bounded domain.
The aim of this section is to find a geometric condition, both simple and general, on the domain $\Omega$ ensuring that its indicator function belongs to a given Besov space. Let us mention that studying the regularity of an indicator function is an important problem in analysis; we refer to the papers by Sickel and coauthors \citep{Sic21} and \citep{Sic24} for recent results on this topic. 
Although Proposition \ref{prop : Espace de Besov indicatrice en dim d} can be deduced from a lemma of Sickel (Lemma 7 of \citep{Sic21}), we choose here to provide a simple proof by wavelets. This approach allowed us to introduce the notion of the Lebesgue boundary, which does not explicitly appear in the aforementioned papers.

The fact that the function $\mathbb{1}_{\Omega}$ and the wavelet $\phi$ have compact support implies that the function $\mathbb{1}_{\Omega}$ has only finitely many nonzero coefficients $c_k.$ Indeed, if the support of the wavelet $\phi_k$ does not intersect $\Omega,$ then the coefficient $c_k$ is zero. Fix $i \in \lbrace 1, \dots, 2^d-1 \rbrace.$ The function $\psi^{(i)}$ also has compact support but, moreover, has zero integral. If $j \ge 0$ is fixed, we therefore observe that the wavelet coefficient $\cijk$ is zero if the support of the wavelet $\psiijk$ does not intersect the boundary of the set $\Omega.$ In fact, in this case, the support of $\psiijk$ is contained either in $\Omega,$ or in $\Omega^{c}.$ We can make this argument more precise: this coefficient is zero as soon as 
$$ \lambda \big( \supp\big(\psiijk\big) \cap  \Omega \big) = 0 \quad \text{ or } \quad \lambda\big( \supp\big(\psiijk\big) \cap \Omega^c \big) = 0,$$
where $\lambda$ denotes the Lebesgue measure on $\R^d.$
This motivates the following definition. 
\begin{defi}
\label{def : frontière de Lebesgue}
Let $\Omega \in \B(\R^d).$ The Lebesgue boundary of $\Omega,$ denoted by $\partial^{\ast} \Omega,$ is the set of points $x \in \R^d$ such that
$$ \forall r > 0 , \quad  \lambda\big( B(x,r) \cap \Omega\big) > 0 \text{ and } \lambda\big( B(x,r) \cap \Omega^c\big) > 0.$$
\end{defi}

A natural question is whether the Lebesgue boundary of a domain is always the usual boundary of some other domain. We shall answer this question positively in Theorem \ref{theo : frontière de Lebesgue} and use this result in the proof of the following theorem.  
In the proposition below, the notation $N_j$ is the one introduced in $\eqref{eq : def de Nj(A)}.$ 

\begin{prop}
\label{prop : Espace de Besov indicatrice en dim d}
Let $\Omega \in \B(\R^d)$ be a bounded Borel set.  
The following statements hold:
\begin{enumerate}[label=\arabic*.]
\item The function $\mathbb{1}_{\Omega}$ belongs to the space $\B^{0}_{\infty,\infty}(\R^d)$.
\item Let $s \in \R$ and $p \in (0, \infty)$. If
\[
\sup_{j \in \N}\, 2^{(sp - d)j}\, N_j(\partial^{\ast} \Omega) < \infty,
\]
then $\mathbb{1}_{\Omega} \in \B^{s}_{p,\infty}(\R^d)$.
\item Let $s \in \R$ and $p, q \in (0, \infty)$. If
\[
\sum_{j = 0}^{\infty} \left( 2^{(sp - d)j} N_j(\partial^{\ast} \Omega) \right)^{\frac{q}{p}} < \infty,
\]
then $\mathbb{1}_{\Omega} \in \Bspq(\R^d)$.
\end{enumerate}
\end{prop}

\begin{proof}
Let $k = (k_1, \dots, k_d) \in \Z^d.$ Since the set $\Omega$ is bounded, there exists an integer $M \in \N$ such that
$$ \Omega \subset [-M,M]^d.$$
Set
$$ K =  \max_{i=1,\dots 2^d-1} (\|\phi\|_1, \|\psi^{(i)}\|_1).$$
Recall that the support of the functions $\phi, \psi^{(1)}, \dots, \psi^{(2^d-1)}$ is the cube $[0,N]^d,$ hence
if the cube $[k, k+N]$ does not intersect the cube $[-M,M]^d,$ then the coefficient $c_k$ is zero. We deduce that the coefficient $c_k$ is nonzero only if
$$ \forall l \in \{1, \dots, d\}, \,  [k_l, k_l +N] \cap [-M,M] \neq \varnothing.$$
We therefore obtain the bound
$$\Card \big( \lbrace k \in \Z^d : c_k \neq 0 \rbrace \big) \le \left( 2M+N+1 \right)^d.$$
From the inequality 
$$|c_k| =  \left|\int_{\R^d} \mathbb{1}_{\Omega}\left(x\right) \phi\left(x-k\right) \dx\right| \le K,$$
we deduce the bounds
$$ \sum_{k\in\Z^d} |c_k|^p \le K^p\, (2M + N +1)^d, \text{ if } p \in (0, \infty) $$
and 
$$ \underset{k\in\Z^d}{\sup} \, |c_k| \le K, \text{ if } p= \infty.$$

Let us now turn to the study of the coefficients $\cijk.$ If $p$ is infinite, for all $i,j \text{ and } k,$ we have 
$$ \big|\cijk\big| = \left| \int_{\R^d} \mathbb{1}_{\Omega}\left( \frac{u+k}{2^j} \right) \psi^{(i)}(u) \d u\right| \le K.$$
The function $\mathbb{1}_{\Omega}$ therefore belongs to the space $\B^{0}_{\infty,\infty}(\R^d),$ which proves the first statement.
Let us now turn to the second one.

Consider the case $p\in (0,\infty).$ Fix $j \in \N.$ The wavelet $\psiijk$ has zero integral and its support is the cube $\kNj,$ hence we deduce that if the intersection between $\partial \Omega$ and $ \kNj$ is empty, then 
$$ \kNj \subset \ring \Omega \qquad \text{ or } \qquad \kNj \cap \overline{\Omega} = \varnothing,$$
and in both cases the coefficient $\cijk$ is zero. The coefficient $\cijk$ is thus nonzero only if the condition 
\begin{equation}
\label{eq : preuve thm indicatrice cond cijk non nul}
\kNj \cap \partial \Omega \neq \varnothing
\end{equation} 
is satisfied.
However, we have noted (see the discussion preceding the definition of the Lebesgue boundary) that the coefficient $\cijk$ is also zero if  
$$\lambda\big( \supp\big(\psiijk\big) \cap \Omega^c \big) = 0 \quad \text{ or } \quad    \lambda \big( \supp\big(\psiijk\big) \cap  \Omega \big) = 0.$$
Now Theorem $\ref{theo : frontière de Lebesgue}$ implies the existence of a set $\widetilde{\Omega} \in \B(\R^d)$ satisfying the following properties 
$$ \lambda\big( \Omega \triangle \widetilde\Omega\big) = 0 \quad \text{ and } \quad \partial \widetilde\Omega = \partial^{\ast} \Omega.$$ 
The functions $\mathbb{1}_{\Omega}$ and $\mathbb{1}_{\widetilde\Omega}$ therefore have the same wavelet coefficients, and condition $\eqref{eq : preuve thm indicatrice cond cijk non nul}$ applied to $\widetilde \Omega$ yields a more precise result: the coefficient $\cijk$ is nonzero only if the condition 
$$ \kNj \cap \partial^{\ast} \Omega \neq \varnothing$$ is satisfied.
With the integers $i$ and $j$ still fixed, we deduce the bounds  
\begin{equation}
\label{eq : preuve theo indicatrice en dim d}
\Card\left(\big\lbrace k \in \Z^d : \cijk \neq 0 \big\rbrace\right) \le \Card\left( \big\lbrace k \in \Z^d : \partial^{\ast} \Omega \cap \kNj \neq \varnothing \big\rbrace\right) \le N^d N_j(\partial^{\ast}\Omega),
\end{equation}
the second bound coming from the fact that $N_j(\partial^{\ast}\Omega)$ counts the dyadic cubes of side $2^{-j}$ that intersect $\partial^{\ast} \Omega.$
From this bound and the same reasoning as before, we deduce the inequality 
$$\sum_{i=1}^{2^d-1} \sum_{k \in \Z^d} \big|\cijk\big|^p \le C N_j(\partial^{\ast}\Omega),$$
where $C = (2^d-1)N^d K^p.$ 
Let now $s\in \R.$ We can now conclude. If there exists a constant $C > 0$ such that for all $j \in \N$ we have
$$ 2^{(sp-d)j}N_j(\partial^{\ast}\Omega) \le C,$$
then $\mathbb{1}_{\Omega} \in \B^{s}_{p,\infty}(\R^d).$
Similarly, if $q \in (0, \infty)$ and if the condition 
$$ \sum_{j=0}^{\infty} \left(2^{(sp-d)j}N_j(\partial^{\ast}\Omega)\right) ^{\frac{q}{p}} < \infty$$
is satisfied, then $\mathbb{1}_{\Omega} \in \Bspq(\R^d).$
\end{proof}

By applying Proposition $\ref{prop : Espace de Besov indicatrice en dim d}$ in the case of interest to us, namely when the exponents $s, p, q$ are equal to $\gamma, 1, 1,$ we obtain the geometric criterion \eqref{eq : indicatrice appartient à B(gamma,1,1)}. 

\subsection{Lebesgue boundary}
Let $\Omega \in \B(\R^d).$ Informally, the Lebesgue boundary of the set $\Omega$ (Definition $\ref{def : frontière de Lebesgue}$) is the set of points that remain on the boundary after adding to or removing from $\Omega$ a set of measure zero. 
This condition is of course consistent with the wavelet definition of Besov spaces \ref{def : espace de Besov globaux}.

We now address the question mentioned after the definition of the Lebesgue boundary, which allowed us to prove the previous theorem, namely whether the Lebesgue boundary is also the usual boundary of some domain. 
Finally, we shall illustrate with an example the fact that the Lebesgue boundary is a more relevant notion for applying the criterion \eqref{eq : indicatrice appartient à B(gamma,1,1)}.

\begin{theo}
\label{theo : frontière de Lebesgue}
Let $\Omega \in \B(\R^d)$ be a Borel set with Lebesgue boundary $\partial^{\ast} \Omega.$ Then, there exists a Borel set $\widetilde\Omega \in \B(\R^d)$ satisfying the following properties: 
$$ \lambda \big( \Omega \triangle \widetilde\Omega \big) = 0  \quad \text{ and } \quad \partial^{\ast} \Omega = \partial \Omegatilde.$$
\end{theo}

\begin{proof}
The idea of the proof is as follows. We begin by partitioning the boundary of $\Omega$ into three sets 
$ \partial^{\ast} \Omega, \partialint\Omega \text{ and } \partialext \Omega$ defined as follows.
For every $x \in \R^d,$
\begin{enumerate}[label = \arabic*.]
\item $ x \in \partial^{\ast} \Omega \quad  \text{ if } \quad \forall r > 0 , \lambda \big( \B(x,r) \cap \Omega\ \big) > 0 \text{ and } \lambda \big( \B(x,r) \cap \Omega^c \big) > 0.$
\item $x \in  \partialint \Omega \quad  \text{ if } \quad x\in \partial \Omega$ and $\exists r > 0, \lambda \big( \B(x,r) \cap \Omega^c  \big) = 0.$
\item $x \in  \partialext \Omega \quad  \text{ if } \quad x\in \partial \Omega$ and $\exists r > 0, \lambda \big( \B(x,r) \cap \Omega  \big) = 0.$
\end{enumerate}
We then define the set $\Omegatilde$ by adding to $\Omega$ the points of $\partialint \Omega$ and removing the points of $\partialext \Omega :$ 
$$\Omegatilde = \left( \Omega \cup \partialint \Omega \right) \setminus \left( \partialext \Omega \right) 
= \left( \Omega \cup \left( \partialint \Omega \cap \Omega^c \right) \right) \setminus \left( \partialext \Omega \cap \Omega \right).$$
After showing that the sets $\partialint \Omega \cap \Omega^c$ and $\partialext \Omega \cap \Omega$ have measure zero, we observe that the symmetric difference of $\Omega$ and $\Omegatilde,$ which is the union of these two sets, has measure zero. The fact that two sets differing by a null set have the same Lebesgue boundary ensures the equality 
$$\partial^{\ast} \Omega = \partial^{\ast} \Omegatilde.$$
To conclude, we show that the two sets $\partialint \Omegatilde$ and $\partialext \Omegatilde$ are empty, which yields the desired equality
$$\partial \Omegatilde = \partial^{\ast} \Omegatilde  =  \partial^{\ast} \Omega.$$ 

We now prove the previous assertions.
First, note the equality
\begin{equation}
\label{eq : preuve frontière de Lebesgue égalités}
\partialext \Omega = \partialint \left(\Omega^c\right)
\end{equation}
highlighting the symmetry between the sets $\Omega$ and $\Omega^c.$ 
The three sets $\partial^{\ast} \Omega, \partialint \Omega$ and  $\partialext \Omega$ are pairwise disjoint subsets of $\partial \Omega,$ and every point belonging to $\partial \Omega$ necessarily belongs to one of these sets, so they form a partition of the boundary of $\Omega.$

We now show that the set $\partialint \Omega \cap \Omega^c$ has measure zero. To this end, introduce the set 
$$ \A = \big\lbrace \B(x,r) : x \in \Q^d, r \in \Q^{\ast}_{+}, \, \lambda\big( \B(x,r) \cap \Omega^c \big) = 0 \big\rbrace.$$
This set being countable, we simply denote by $\B_i,$ with $i \in \N,$ the balls that compose it.
The inclusion 
$$ \partialint \Omega \cap \Omega^c \subset \bigcup_{i \in \N}\left( \B_i \cap \Omega^c \right)$$
ensures that the set $\partialint \Omega \cap \Omega^c$ is included in a countable union of sets of measure zero.
Indeed, each set in this union has measure zero by definition of $\A.$ This therefore ensures that $\partialint \Omega \cap \Omega^c$ has measure zero.  By replacing the set $\Omega$ with $\Omega^c$ and using equality \eqref{eq : preuve frontière de Lebesgue égalités},
we also deduce that the set $\partialext \Omega \cap \Omega$ has measure zero.

The equality 
$$ \Omega \triangle \widetilde\Omega = \left( \partialint \Omega \cap \Omega^c \right) \cup \left(\partialext \Omega \cap \Omega \right) $$
thus implies that the set $\Omega \triangle \Omegatilde$ has measure zero.
We deduce the equality
$$ \lambda \big( \B(x,r) \cap \Omega \big) = \lambda \big( \B(x,r) \cap \Omegatilde \big),$$ for all $x \in \R^d$ and $ r > 0.$
The same equality holds if we replace $\Omega$ with $\Omega^c,$ and we deduce that the sets $\Omega$ and $\Omegatilde$ have the same Lebesgue boundary 
$$ \partial^{\ast} \Omega = \partial^{\ast} \Omegatilde.$$
It remains to show the equality 
$$ \partial^{\ast} \Omegatilde = \partial \Omegatilde,$$
which is equivalent to the fact that the sets
$\partialint \Omegatilde$ and $\partialext \Omegatilde$ are empty. 
To show that the set $\partialint \Omegatilde$ is empty, we prove that every point $x \in \R^d$ satisfying the following property
 $$\exists r > 0,\, \lambda\big( \B(x,r) \cap \Omega^c \big) = 0$$
 necessarily belongs to the interior of $\widetilde \Omega.$
So let $x \in \R^d$ be such a point.
We now show that the ball $\B(x,r)$ is contained in $\Omegatilde.$
Let $y \in \B(x,r).$ We first observe that $y \in \overline{\Omega}.$
We now show that $y \in \Omegatilde.$ 
The point $y$ does not belong to the Lebesgue boundary of the set $\Omega,$ hence it necessarily belongs to one of the three disjoint sets 
$$ \ring \Omega, \quad \partialint \Omega \quad \text{ or } \quad \partialext \Omega.$$
However, recall that the set $\Omegatilde$ is defined by 
$$ \Omegatilde = \left( \Omega \cup \partialint \Omega \right) \setminus \partialext \Omega,$$
so it suffices to show that the point $y$ does not belong to $\partialext \Omega$ in order to conclude that it belongs to $\Omegatilde.$ This follows from the fact that for $\epsilon$ small enough, if $y$ belonged to $\partialext \Omega,$ we would have both 
$$ \lambda\big( \B(y,\epsilon) \cap \Omega \big) = 0 \quad \text{ and } \quad  \lambda\big( \B(y,\epsilon) \cap \Omega^c \big) = 0, $$ 
which is absurd.
Thus, the ball $\B(x,r)$ is contained in the set $\Omegatilde,$ and we deduce 
that the point $x$ belongs to the interior of $\Omegatilde,$ and that the set $\partialint \Omegatilde$ is empty.
Finally, we deduce from the equality
$$  \left(\Omegatilde\right)^c = \widetilde{\big(\Omega^c\big)}$$
and from the symmetry between the sets $\widetilde \Omega$ and $\Omegatilde^c$ (see \eqref{eq : preuve frontière de Lebesgue égalités}),
the equality
$$ \partialext \widetilde \Omega = \partialint \big(\Omegatilde^c\big)$$
which implies that the set $\partialext \Omegatilde$ is also empty.
\end{proof}
We now illustrate the advantage of considering the Lebesgue boundary rather than the usual boundary.

Let $\Omega \subset \R^2$ be the bounded domain shown in Figure \ref{fig:objet_1}, delimited by the black curve, where the red curve is added and the purple curve is removed. Since the boundary of $\Omega$ is the union of the black, red and purple curves, its upper box dimension is strictly greater than $1$ (and can be chosen arbitrarily close to $2$). However, the Lebesgue boundary of the set $\Omega$ is the black curve, illustrated in Figure \ref{fig:objet_2}, of dimension $1.$

\begin{figure*}[htbp]
\centering
\begin{subfigure}[t]{0.48\linewidth}
  \centering
  \includegraphics[width=\linewidth]{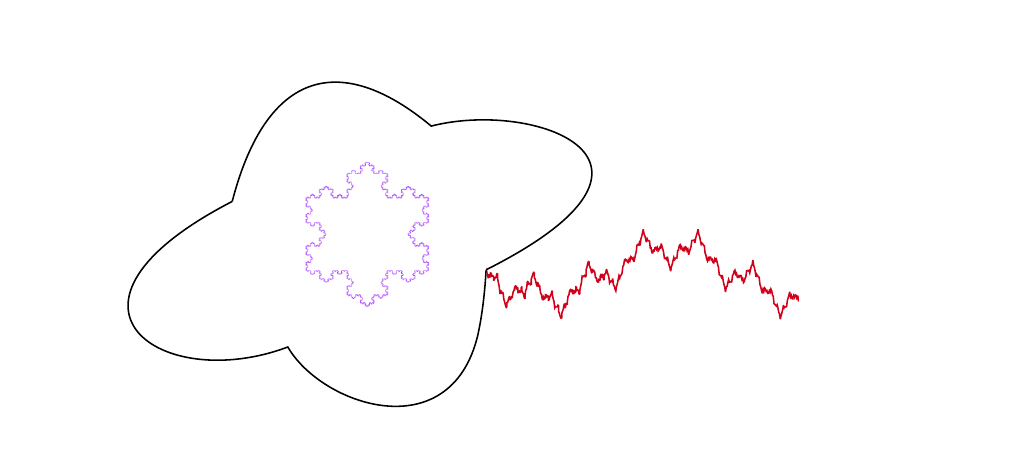}
  \caption{Boundary of $\Omega$.}
  \label{fig:objet_1}
\end{subfigure}\hfill
\begin{subfigure}[t]{0.48\linewidth}
  \centering
  \includegraphics[width=\linewidth]{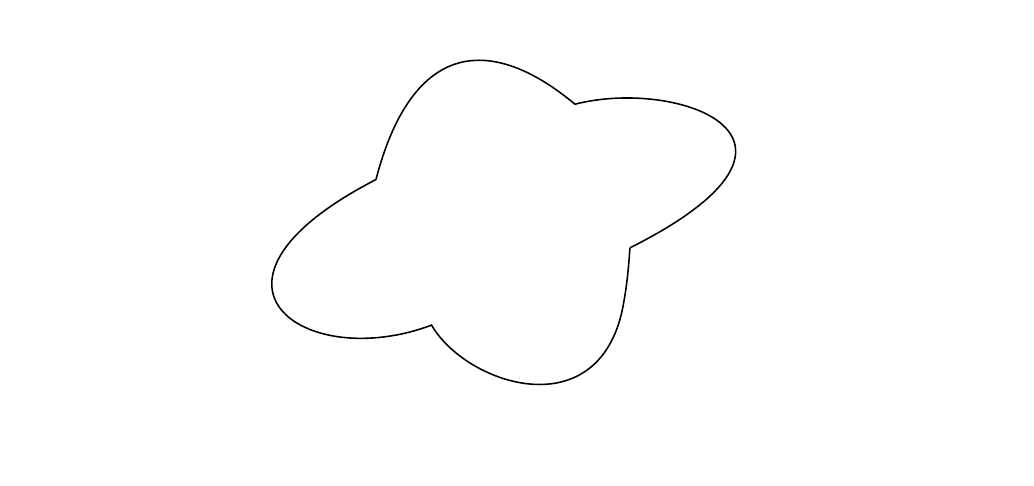}
  \caption{Lebesgue boundary of $\Omega$.}
  \label{fig:objet_2}
\end{subfigure}
\caption{Boundary and Lebesgue boundary of a fractal domain in dimension~2.}
\label{fig:objet_3}
\end{figure*}

\subsection{Upper box dimension and examples}
\label{section : graphe fonction hölder}
The purpose of this section is to illustrate the relevance of the criterion \eqref{eq : indicatrice appartient à B(gamma,1,1)}. We begin by comparing it with a condition involving the upper box dimension of the boundary. We then provide an example of a set showing that this criterion is more general.
 
For any subset $A$ of $\R^d,$ recall the notation $N_j(A)$ defined in $\eqref{eq : def de Nj(A)}.$
\begin{defi}[Upper box dimension]
\label{defi : dim boite sup}
For any nonempty bounded subset $A$ of $\R^d$, the \textit{upper box dimension} of $A$ is defined by
\[
\dimb(A) = \limsup_{j \to \infty} \; \frac{\log N_j(A)}{\log 2^j}.
\]
\end{defi}

Recall that the upper box dimension of a bounded set $A$ is a real number belonging to the interval $[0,d].$ When $A$ is a smooth submanifold of $\R^d$, it coincides with the usual notion of dimension. For instance, the upper box dimension of a non-degenerate rectangle in $\R^d$ is $d,$ while the upper box dimension of the sphere is $d-1.$ 

\begin{cor}
\label{cor : dim boîte indicatrice appartient à B(gamma,1,1)}
Let $\Omega \subset \R^d$ be a set with Lebesgue boundary $\partial^{\ast}\Omega$.  
Let $\gamma \in (0,1)$ and set $\beta = d - \gamma$. Then,
\[
\dimb(\partial^{\ast}\Omega) < \beta 
\implies 
\sum_{j=0}^{\infty} 2^{-\beta j} N_j(\partial^{\ast}\Omega) < \infty.
\]
\end{cor}

\begin{proof}
Let $\delta \in (\dimb(\partial^{\ast}\Omega), \beta)$.  
The condition $\dimb(\partial^{\ast}\Omega) < \delta$ implies the existence of a constant $C > 0$ such that, for all $j \in \N$,
\[
N_j(\partial^{\ast}\Omega) \le C\, 2^{\delta j}.
\]
We then deduce that the series $\sum_{j \in \N} 2^{-\beta j} N_j(\partial^{\ast}\Omega)$ converges geometrically.
\end{proof}
In particular, since the Lebesgue boundary is contained in the usual boundary, we recover the condition of Alberti, Stepanov and Trevisan in dimension $2$ stated in \citep{AlSTTRE24}.

We now show that criterion \eqref{eq : indicatrice appartient à B(gamma,1,1)} provides a more precise condition than the previous one supplied in Corollary \ref{cor : dim boîte indicatrice appartient à B(gamma,1,1)} which involves the upper box dimension. We restrict ourselves to the case of dimension $2.$  

Let $\beta \in (1,2).$ We will construct a set $\Omega \subset \R^2$ such that 
\begin{equation}
\label{eq : hyp graphe fonction}
 \dimb(\partial^{\ast} \Omega) = \beta \quad \text{ and } \quad \sum_{j=0}^{\infty} 2^{-\beta j} N_j(\partial^{\ast} \Omega) < \infty.
\end{equation}
The domain under consideration will be the epigraph of a Hölder function $g_{\beta}$ defined on $[0,1],$ which we will construct explicitly via its Schauder series. Let us begin with some reminders about this basis.

Let $\Lambda : [0,1] \rightarrow \R$ be the hat function defined by 
$$\ \Lambda(x)  = \begin{cases} x & \text{ if } x \in \left[0,\tfrac{1}{2}\right],
 \\ 1-x &\text{ if } x \in \left[\tfrac{1}{2},1\right]. \end{cases}$$
For all integers $j \ge 0 $ and $k \in \{0, \dots, 2^j -1\},$ we denote by $\Lambda_{j,k} : [0,1] \longrightarrow \R$ the function defined by 
 $$\ \Lambda_{j,k}(x)  = \begin{cases} \Lambda(2^jx-k) & \text{ if } x \in \left[\tfrac{k}{2^j},\tfrac{k+1}{2^j}\right],
 \\ 0 &\text{ otherwise}. \end{cases}$$
The Schauder basis of the space $\C([0,1], \R)$ is the family of functions
\[
\big\{\, 1,\, x,\, \Lambda_{j,k}(x) : j \ge 0,\; k \in \{0, \dots, 2^j - 1\} \,\big\}.
\]
Recall the following theorem.

\begin{theo}
\label{theo : Base de Schauder}
Let $f : [0,1] \longrightarrow \R$ be a continuous function.  
Then,
\[
f(x) = f(0) + (f(1) - f(0))\,x + \sum_{j=0}^{\infty} \sum_{k=0}^{2^j - 1} c_{j,k}\, \Lambda_{j,k}(x),
\]
where
\[
c_{j,k} = 2 f\!\left(\tfrac{k + 1/2}{2^j}\right)
- f\!\left(\tfrac{k + 1}{2^j}\right)
- f\!\left(\tfrac{k}{2^j}\right),
\]
and the series converges uniformly to $f$ on $[0,1]$.

Moreover, for any $\gamma \in (0,1)$, if there exists a constant $C > 0$ such that, for all $j \ge 0$ and $k \in \{0, \dots, 2^j - 1\}$, one has
\[
|c_{j,k}| \le C\, 2^{-\gamma j},
\]
then the function $f$ is $\gamma$-Hölder on $[0,1]$.
\end{theo}

We now introduce some notation.  
Let $f : [0,1] \longrightarrow \R$ be a continuous function, and let $j \ge 0$ be an integer.  
Denote by $N_j(f)$ the number of dyadic squares of generation $j$ required to cover the graph of $f$. Given a dyadic interval $\lambda$ of generation $j$, i.e.
\[
\lambda = \left[ \tfrac{k}{2^j}, \tfrac{k+1}{2^j} \right),
\qquad k \in \{0, \dots, 2^j - 1\},
\]
let $N_j^{\lambda}(f)$ denote the number of dyadic squares needed to cover the graph of $f$ restricted to $\lambda$, and let $\Osc(f, \lambda)$ be the oscillation of $f$ on $\lambda$, defined by
\[
\Osc(f, \lambda)
= \sup_{x \in \lambda} f(x)
- \inf_{x \in \lambda} f(x).
\]
We clearly have the inequality
\begin{equation}
\label{eq : encadrement de Nj par osc}
2^j \Osc(f,\lambda)
\le N_j^{\lambda}(f)
\le 2^j \Osc(f,\lambda) + 2.
\end{equation}

Let $\gamma \in (0,1)$ and $\delta \in \R$.  
Consider the function $f_{\gamma, \delta} : [0,1] \longrightarrow \R$ defined by
\[
f_{\gamma, \delta}(x)
= \sum_{j=1}^{\infty} \sum_{k=0}^{2^j - 1} 2^{-\gamma j}\, j^{\delta} \Lambda_{j,k}(x).
\]
The function $g_{\beta}$ that we aim to construct will be of this type.  
The function $f_{\gamma, \delta}$ is a continuous and positive function whose Schauder coefficients are given, for all $j \ge 1$ and $k \in \{0, \dots, 2^j - 1\}$, by
\begin{equation}
\label{eq : exemple fonction 1}
2^{-\gamma j} j^{\delta}
= 2 f_{\gamma, \delta}\!\left(\tfrac{k + 1/2}{2^j}\right)
- f_{\gamma, \delta}\!\left(\tfrac{k + 1}{2^j}\right)
- f_{\gamma, \delta}\!\left(\tfrac{k}{2^j}\right).
\end{equation}
Let $j \ge 1$, and let $\lambda$ be a dyadic interval of generation $j$.  
From the previous equality for the Schauder coefficients of $f_{\gamma, \delta}$, we deduce the lower bound
\begin{equation}
\label{eq : exemple fonction 2}
2^{-\gamma j} j^{\delta} \le 2\, \Osc(f_{\gamma, \delta}, \lambda).
\end{equation}
Inserting this inequality into \eqref{eq : encadrement de Nj par osc}, we obtain
\begin{equation}
\label{eq : exemple fonction 3}
2^{(1 - \gamma)j - 1} j^{\delta}
\le N_j^{\lambda}(f_{\gamma, \delta})
\le 2^j \Osc(f_{\gamma, \delta}, \lambda) + 2.
\end{equation}
We now estimate $\Osc(f_{\gamma, \delta}, \lambda)$.  
In what follows, $C$ denotes a strictly positive constant depending only on $\gamma$ and $\delta$, whose value may change from line to line.  
For $x, y \in \lambda$, we have
\begin{align*}
|f(x) - f(y)|
&\le \sum_{j' < j} \sum_{k'=0}^{2^{j'} - 1} 2^{-\gamma j'} {j'}^{\delta}
  \big| \Lambda_{j',k'}(x) - \Lambda_{j',k'}(y) \big| \\
&\quad + \sum_{j' \ge j} \sum_{k'=0}^{2^{j'} - 1} 2^{-\gamma j'} {j'}^{\delta}
  \big( \Lambda_{j',k'}(x) + \Lambda_{j',k'}(y) \big).
\end{align*}
We now bound each of the two terms on the right-hand side of this inequality.

We begin with the first term.  
Let $1 \le j' < j$.  
The interval $\lambda$ is contained in a unique dyadic interval of generation $j'$, so there exists a unique $k \in \{0, \dots, 2^{j'} - 1\}$ such that the quantity $|\Lambda_{j',k}(x) - \Lambda_{j',k}(y)|$ is nonzero.  
The inequality
\[
|\Lambda_{j',k}(x) - \Lambda_{j',k}(y)| \le 2^{j'} |x - y|
\]
therefore yields
\begin{align*}
\sum_{j' < j} \sum_{k' = 0}^{2^{j'} - 1} 2^{-\gamma j'} {j'}^{\delta}
  |\Lambda_{j',k'}(x) - \Lambda_{j',k'}(y)|
&= \sum_{j' < j} 2^{-\gamma j'} {j'}^{\delta}
  |\Lambda_{j',k}(x) - \Lambda_{j',k}(y)| \\
&\le \sum_{j' < j} 2^{(1 - \gamma) j'} {j'}^{\delta} |x - y| \\
&\le C\, 2^{(1 - \gamma) j} j^{\delta} |x - y| \\
&\le C\, 2^{-\gamma j} j^{\delta},
\end{align*}
where the last inequality follows from the fact that $|x - y| \le 2^{-j}$.

We now bound the second term.  
For every $j' \ge j$, there exists a unique dyadic interval of generation $j'$ containing $x$.  
Since the function $\Lambda$ is bounded by $\tfrac{1}{2}$, we obtain
\[
\sum_{j' \ge j} \sum_{k' = 0}^{2^{j'} - 1} 2^{-\gamma j'} {j'}^{\delta}
  \Lambda_{j',k'}(x)
\le \frac{1}{2} \sum_{j' \ge j} 2^{-\gamma j'} {j'}^{\delta},
\]
and the same bound holds for $y$.  
We deduce
\begin{align*}
\sum_{j' \ge j} \sum_{k' = 0}^{2^{j'} - 1} 2^{-\gamma j'} {j'}^{\delta}
  \big( \Lambda_{j',k'}(x) + \Lambda_{j',k'}(y) \big)
&\le \sum_{j' \ge j} 2^{-\gamma j'} {j'}^{\delta} \\
&\le C\, 2^{-\gamma j} j^{\delta}.
\end{align*}
From these two bounds, we obtain the desired inequality
\[
\Osc(f_{\gamma, \delta}, \lambda)
\le C\, 2^{-\gamma j} j^{\delta}.
\]
Inserting this inequality into \eqref{eq : exemple fonction 3}, and using the fact that $1 - \gamma$ is strictly positive, we obtain
\[
2^{(1 - \gamma) j - 1} j^{\delta}
\le N_j^{\lambda}(f_{\gamma, \delta})
\le C\, 2^{(1 - \gamma) j} j^{\delta}.
\]
From the equality
\[
N_j(f_{\gamma, \delta})
= \sum_{k = 0}^{2^j - 1}
  N_j^{\lambda_{j,k}}(f_{\gamma, \delta}),
\qquad
\lambda_{j,k}
= \left[\tfrac{k}{2^j}, \tfrac{k + 1}{2^j}\right),
\]
we deduce
\begin{equation}
\label{eq : exemple fonction encadrement Njlambda}
2^{(2 - \gamma) j - 1} j^{\delta}
\le N_j(f_{\gamma, \delta})
\le C\, 2^{(2 - \gamma) j} j^{\delta}.
\end{equation}

We now set $g_{\beta} = f_{2 - \beta, -2}$.  
We have shown the existence of constants $C_1, C_2 > 0$ such that, for every $j \in \N^{\ast}$,
\[
\frac{C_1}{j^2}\, 2^{\beta j}
\le N_j(g_{\beta})
\le \frac{C_2}{j^2}\, 2^{\beta j}.
\]
We deduce that
\[
\sum_{j = 1}^{\infty} 2^{-\beta j} N_j(g_{\beta}) < \infty,
\]
and that the upper box dimension of the graph of the function $g_{\beta}$ is equal to $\beta$.  
Finally, we note that the function $g_{\beta}$ is $(2 - \beta)$-Hölder by Theorem~\ref{theo : Base de Schauder}, and it is easily checked that the Lebesgue boundary of the epigraph of a Hölder function coincides with the graph of such a function.  
We therefore deduce that the set
\[
\Omega = \big\{ (x, y) : x \in [0,1],\ 0 \le y \le g_{\beta}(x) \big\}
\]
satisfies the desired properties~\eqref{eq : hyp graphe fonction}.

\begin{figure*}[htbp]
\centering
\includegraphics[scale=0.7]{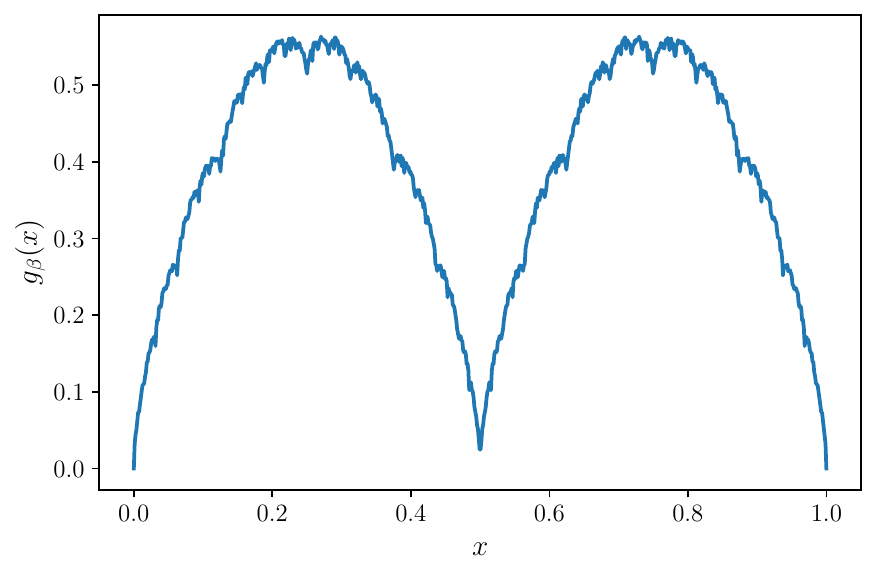}
\caption{Graph of the function $g_{\beta}$ for $\beta = 1.99$.}
\label{fig:graphe holder}
\end{figure*}

\section{Towards more general assumptions}
\label{section : Towards more general assumptions}
Let $\alpha, \beta_1, \dots, \beta_d$ be real numbers satisfying \eqref{eq : alpha + beta > d}, and let $\beta, \gamma$ be the real numbers defined in \eqref{eq : def de beta et gamma}.
The purpose of this section is to generalize Theorem $\ref{theo : distribution globale f dg1 ^....^ dgd}$ by considering weaker Hölder regularity assumptions. 

We first focus on the case where the functions are still globally Hölder continuous, and we emphasize the fact that the function $f$ and the functions $g^1, \dots, g^d$ play different roles in the construction of the distribution $\fdgd.$ More precisely, we show that it suffices to assume that the functions $g^1, \dots, g^d$ belong to the homogeneous Hölder spaces $\dot\C^{\beta_1}(\R^d), \dots, \dot\C^{\beta_d}(\R^d),$ defined at the beginning of Section \ref{section : espace de hölder homogène}, in order for the distribution $\fdgd$ to be well-defined in the space of globally Hölder distributions $\C^{-\gamma}(\R^d).$
Since the continuity of $\M$ in Theorem~\ref{theo : distribution globale f dg1 ^....^ dgd} 
with respect to the functions $g^1, \dots, g^d$ is expressed exclusively in terms of their 
Hölder seminorms, we show that the mapping constructed on the space 
\[
\Calpha(\R^d) 
\times \dot\C^{\beta_1}(\R^d) 
\times \dots 
\times \dot\C^{\beta_d}(\R^d)
\]
is also continuous. Moreover, it is still possible to define, by 
\eqref{eq : Int_Omega fdgd = < fdgd, 1_Omega> }, the integral 
\[
\int_{\Omega} \fdgd,
\]
for any Borel set $\Omega$ such that 
$\mathbb{1}_{\Omega} \in \B^{\gamma}_{1,1}(\R^d).$

Next, we consider the case where the functions $f, g^1, \dots, g^d$ are only assumed to be locally Hölder continuous on $\R^d:$ 
$$f\in \Calpha_{\loc}(\R^d), g^1 \in \C^{\beta_1}_{\loc}(\R^d), \dots, g^d \in \C^{\beta_d}_{\loc}(\R^d).$$
We refer to Section \ref{section : Construction et étude de la distribution : cas localement höldérien} for a precise definition of these spaces.
Under these assumptions, it is still possible to define the distribution $\fdgd,$ but this distribution now belongs to a space of locally Hölder distributions (see the characterization \eqref{eq : distribution localement hölder}). 
We then prove Theorem \ref{theo : distribution locale f dg1 ^....^ dgd}, which is a local version of Theorem \ref{theo : distribution globale f dg1 ^....^ dgd}. Finally, we use a standard localization argument with cut-off function to define the integral 
$\int_{\Omega} \fdgd,$
by \eqref{eq : Int_Omega fdgd = < fdgd, 1_Omega> cas local}, 
for any bounded Borel set $\Omega \in \B(\R^d)$ satisfying condition \eqref{eq : condition besov pour l'intégral}.
\subsection{Globally Hölder case} 
\label{section : espace de hölder homogène}
We begin by recalling some definitions and notations regarding homogeneous Hölder spaces.
Let 
$$ \Halpha(\R^d) = \lbrace f: \R^d\longrightarrow \R : \|f\|_{\alpha} < \infty \rbrace$$
denote the space of real $\alpha$-Hölder functions defined on $\R^d.$
Let 
$$ \Calphap(\R^d) = \Halpha(\R^d) / \R$$
denote the quotient space of $\Halpha(\R^d)$ by the subspace of constant functions. With a slight abuse of terminology, we shall still refer to the elements of $\Calphap(\R^d)$ as functions.
For any function $\dot f\in \Calphap(\R^d),$ we may define  
$$ \|\dot f\|_{\Calphap} = \|f\|_{\alpha},$$
where $f\in \Halpha(\R^d)$ is any representative of $\dot f.$  
We recall that the space $\left(\Calphap(\R^d), \|\cdot\|_{\Calphap} \right)$ is a Banach space. We say that a function $\dot f$ is Lipschitz if it admits a Lipschitz representative. Since the partial derivatives of such a function are always well-defined, we can freely refer to the Jacobian of Lipschitz functions $g^1 \in \dot\C^{\beta_1}(\R^d) , \dots, g^d \in \C^{\beta_d}(\R^d).$

We now have all the tools to derive a more general version of Theorem $\ref{theo : distribution globale f dg1 ^....^ dgd}$ involving homogeneous Hölder spaces.
In the remainder of this section, we fix functions  
 $$ f\in \Calpha(\R^d),  g^1 \in \H^{\beta_1}(\R^d), \dots,\,  g^d \in \H^{\beta_d}(\R^d).$$
For any rectangle $R \subset \R^d,$ Theorem $\ref{theo : intégrale de Zust}$ allows us to define the Züst integral: 
$$\int_{R} \fdgd.$$
Except for the fact that a Lipschitz function on $\R^d$ is not necessarily Hölder on $\R^d,$ everything works exactly as in the case where the functions $g^1, \dots, g^d$ belong to the inhomogeneous Hölder spaces $\C^{\beta_1}(\R^d), \dots, \C^{\beta_d}(\R^d).$ We can therefore still define the distribution $\fdgd$ by \eqref{eq : def de la distribution fdg1^...^dgd}, and, by the same argument as in Sections \ref{section : Notations globales} and \ref{section : régularité de la distribution globale}, construct the $(d+1)$-linear map  
\begin{equation}
\label{eq : def de l'appli M sur holder homogène}
\begin{array}{rrcl}
\M :  &\Calpha(\R^d) \times  \H^{\beta_1}(\R^d) \times \dots \times \H^{\beta_d}(\R^d) & \longrightarrow & \C^{-\gamma}(\R^d)  \\
& (f, g^1, \dots, g^d) & \longmapsto &  \fdgd
\end{array}
\end{equation}
satisfying the following properties:
\begin{enumerate}[label=\arabic*.]
\item Let $g^1 \in \H^{\beta_1}(\R^d) , \dots, g^d \in \H^{\beta_d}(\R^d).$ If the functions $g^1, \dots, g^d$ are Lipschitz, then 
$$\fdgd = f\det(\dg).$$
\item Let $\alpha' \le \alpha$ and $\beta_i'\le \beta_i,$ $i \in \{1,\dots,d\},$ satisfying \eqref{eq : alpha + beta > d}. Set $\gamma' = d-\sum_{i=1}^d \beta_i'.$ 
Then there exists a constant $C >0$ such that, for all functions $f \in \C^{\alpha}(\R^d),$ $ g^i \in \H^{\beta_i}(\R^d),$ we have 
$$\|\fdgd\|_{\C^{-\gamma'}} \le C\, \|f\|_{\C^{\alpha'}} \prod_{i=1}^d\|g^i\|_{\beta_i'}.$$
\end{enumerate}
Note that the restriction of this application to the space 
$\Calpha(\R^d) \times  \C^{\beta_1}(\R^d) \times \dots \times \C^{\beta_d}(\R^d)$ coincides with the application defined in Theorem \ref{theo : distribution globale f dg1 ^....^ dgd}. We have therefore chosen to give them the same name. 
The uniqueness of this application is proved in the same way as in Theorem \ref{theo : distribution globale f dg1 ^....^ dgd} (see the end of Section \ref{section : Unicité cas global}), using a variant of the approximation lemma \ref{lem : lemme d'approximation global}.
\begin{lemme}[Global approximation lemma $2$]
\label{lem : lemme d'approximation global 2}
Let $\alpha, \alpha' \in (0,1) $ with $\alpha' < \alpha.$ For any function $f\in \H^{\alpha}(\R^d),$ there exists a sequence of real functions $(f_n)_{n\in\N}$ defined on $\R^d$ satisfying the following properties.
\begin{enumerate}[label=\arabic*.]
\item For every integer $n\in \N^{\ast},$ the function $f_n$ is of class $\C^{\infty}$ on $\R^d$ and we have
$$ \|f_n\|_{\alpha} \le \|f\|_{\alpha}.$$
\item The sequence $(f_n)_{n\in\N}$ converges uniformly to the function $f$ and
$$\underset{n \to \infty}{\lim} \|f-f_n\|_{\alpha'}  = 0.$$
\end{enumerate}
\end{lemme}
Finally, let us note that the properties stated in Proposition \ref{prop : prop de l'appli M global} are also satisfied by the application $\M.$  
This shows that this application is compatible with the quotient, and we thus have proved the desired theorem.

\begin{theo}
\label{theo : distribution globale f dg1 ^....^ dgd homogène}
There exists a unique $(d+1)$-linear map 
$$\begin{array}{rrcl}
\M :  &\Calpha(\R^d) \times  \dot\C^{\beta_1}(\R^d) \times \dots \times \dot \C^{\beta_d}(\R^d) & \longrightarrow & \C^{-\gamma}(\R^d)  \\
& (f, g^1, \dots, g^d) & \longmapsto &  \fdgd
\end{array}$$
satisfying the following properties:
\begin{enumerate}[label=\arabic*.]
\item Let $g^1 \in \dot\C^{\beta_1}(\R^d) , \dots, g^d \in \dot\C^{\beta_d}(\R^d).$  
If the functions $g^1, \dots, g^d$ are Lipschitz, then
$$ \fdgd = f\det(\dg).$$

\item Let $\alpha' \le \alpha$ and $\beta_i'\le \beta_i,$ $i \in \{1,\dots,d\},$ satisfying \eqref{eq : alpha + beta > d}.  
Set $\gamma' = d-\sum_{i=1}^d \beta_i'.$  
Then, there exists a constant $C >0$ such that, for all functions $f \in \C^{\alpha}(\R^d)$ and $ g^i \in \dot \C^{\beta_i}(\R^d),$ we have 
$$\|\fdgd\|_{\C^{-\gamma'}} \le C\, \|f\|_{\C^{\alpha'}}\prod_{i=1}^d\|g^i\|_{\dot\C^{\beta_i'}}.$$
\end{enumerate}
\end{theo}

\subsection{Locally Hölder case}
\label{section : Construction et étude de la distribution : cas localement höldérien}
The goal of this section is to study the case where the functions are only assumed to be locally Hölder on $\R^d$.  
We begin with a few reminders on local Hölder spaces.  

For any subset $A \subset \R^d$, we set
\[
\|f\|_{\alpha, A}
= \sup_{\substack{x, y \in A \\ x \neq y}}\;
  \frac{|f(y) - f(x)|}{|y - x|^{\alpha}},
\qquad
\|f\|_{\infty, A}
= \sup_{x \in A}\; |f(x)|,
\]
and we define the seminorm
\begin{equation}
\label{eq : semi-norme fonction hölder}
\|f\|_{\Calpha(A)} = \|f\|_{\infty, A} + \|f\|_{\alpha, A}.
\end{equation}
We denote
$$\Calphaloc(\R^d) = \lbrace f : \R^d \longrightarrow \R : \forall n \in \N, \, \|f\|_{\alpha,\nn} < \infty \rbrace,$$
the space of locally $\alpha$-Hölder functions on $\R^d.$  
This space, endowed with the family of seminorms 
$\{ \|\cdot\|_{\Calpha(\nn)} : n \in \N \}$, is a Fréchet space.

We now recall the general classical method for defining a local function space (it is straightforward to check that this method coincides with the previous definition in the case of Hölder functions).
\begin{defi}
\label{def : espace de Besov locaux}
Let $f \in \D'(\R^d).$ We say that the distribution $f$ belongs to the local Besov space $\Bspqloc(\R^d)$ if
$$ \forall \phi \in \D(\R^d), \qquad f\cdot\phi \in \Bspq(\R^d).$$
\end{defi}

It is easy to check that local Besov spaces can be characterized through the following condition on wavelet coefficients.
\begin{prop}
Let $s \in \R$ and $ p, q \in (0,\infty].$  
A distribution $f \in \D'(\R^d),$ of order at most $r,$ belongs to the space $\Bspqloc(\R^d)$ if, for every integer $n \in \N,$ its wavelet coefficients satisfy
$$\sum_{|k| \le n} |c_k|^p < \infty \quad  \text{and} \qquad 
\sum_{j=0}^{\infty} \left(\sum_{|k2^{-j}| \le n} \sum_{i=1}^{2^d-1} 
 \left| 2^{(s-\frac{d}{p})j}\cijk\right|^p \right)^{\frac{q}{p}} <\infty,$$
with the usual modifications if $p$ or $q$ is infinite.
\end{prop}

For every integer $n \in \N,$ we denote by $\|f\|_{\Bspq(\nn)}$ the seminorm
\begin{equation}
\label{eq : semi-norme Bspq}
\left( \sum_{|k|\le n} |c_k|^p \right)^{\frac{1}{p}} +  \left( \sum_{j=0}^{\infty} \left( \sum_{|k2^{-j}| \le n} \sum_{i=1}^{2^d-1} 
 \left| 2^{(s-\frac{d}{p})j} \cijk \right|^p \right)^{\frac{q}{p}} \right)^{\frac{1}{q}}.
\end{equation}
Let $\gamma \in \R \setminus \N$.  
In particular, by taking $p = q = \infty$, we obtain the following characterization of the space $\Cgammaloc(\R^d)$:
\begin{equation}
\label{eq : distribution localement hölder}
\forall n \in \N, \; \exists C_n > 0 \text{ such that } 
\quad
\sup_{|k| \le n} |c_k| \le C_n
\quad \text{and} \quad
\sup_{j \ge 0} \,
\sup_{\substack{|k 2^{-j}| \le n \\ i \in \{0, \dots, 2^d - 1\}}}
|\cijk| \le C_n.
\end{equation}

We now have all the tools needed to prove Theorem \ref{theo : distribution locale f dg1 ^....^ dgd}. 
Recall that this theorem is more general than Theorem \ref{theo : distribution globale f dg1 ^....^ dgd}. Indeed, let $\M$ denote the mapping defined in Theorem \ref{theo : distribution globale f dg1 ^....^ dgd} and temporarily let $\M_{\loc}$ denote that of Theorem \ref{theo : distribution locale f dg1 ^....^ dgd}. Then the restriction of the mapping $\M_{\loc}$ to the space $\Calpha(\R^d) \times \C^{\beta_1}(\R^d) \times \dots \times \C^{\beta_d}(\R^d)$ coincides with the mapping $\M.$ We therefore choose to give them the same name. 
To prove this theorem, we follow the main steps of the proof of Theorem \ref{theo : distribution globale f dg1 ^....^ dgd} and detail only the arguments that differ from the globally Hölder case.

Let $\alpha, \beta_1, \dots, \beta_d$ be real numbers satisfying condition $\eqref{eq : alpha + beta > d}.$ In what follows, we fix locally Hölder functions
$$f\in \Calpha_{\loc}(\R^d), g^1 \in \C^{\beta_1}_{\loc}(\R^d), \dots, g^d \in \C^{\beta_d}_{\loc}(\R^d).$$

Let $R = [a_1, b_1] \times \dots \times [a_d, b_d]$ be a rectangle in $\R^d.$  By Theorem $\ref{theo : intégrale de Zust},$ it is possible under these assumptions to define Züst’s integral: 
$$\int_{R} \fdgd.$$
We begin by recalling the sewing inequality \eqref{eq : inégalité de couture} in this context. For all $\alpha' \in (0, \alpha]$,  $\beta_1'\in (0, \beta_1]$, $\dots$, $\beta_d'\in (0, \beta_d]$ satisfying $\eqref{eq : alpha + beta > d},$ we have 
\begin{equation}
\label{eq : borne de couture dans le cas localement hölder}
\left| \int_{R} \fdgd - f(a_1, \dots, a_d) \int_{R} \dgd \right| \le C\,  \delta(R)^{\alpha' + \beta'}\, \|f\|_{\alpha', R} \prod_{i=1}^d \|g^i\|_{\beta_i', R},
\end{equation}
where $\beta' = \sum_{i=1}^d \beta_i'$, and  the constant $C$ depends only on $d, \alpha', \beta_1', \dots, \beta_d'.$
Since Lemma \ref{lem : prolongement de l'intégrale quand la fonction vaut 0} and Proposition \ref{prop : fdg forme linéaire sur Calpha} adapt without difficulty to the locally Hölder case, we can still define the distribution $\fdgd$ by \eqref{eq : def de la distribution fdg1^...^dgd} in this setting.
Let us mention that  we can only deduce from \eqref{eq : borne de couture dans le cas localement hölder} the following inequality: for any compact $K \subset \R^d$ and any test function $\phi \in \D(K),$ we have
$$ | \langle \fdgd, \phi \rangle | \le C_K\, \|f\|_{\Calpha(K)} \prod_{i=1}^d \|g_i\|_{\beta_i,K}\; \|\phi\|_{\Calpha}.$$
When the functions $g^1, \dots, g^d$ are locally Hölder, an argument similar to that of Lemma \ref{lem : f dg1 ^...^dgd = fdet(dg) globalement hölder} shows that the distribution $\fdgd$ coincides with the locally essentially bounded function $f\det(\dg).$ We now turn to the study of the regularity of this distribution.

\begin{prop}
\label{prop : régularité de la distribution localement höldérien}
Let $\alpha' \in (0, \alpha]$ and $\beta_1', \dots, \beta_d' \in (0, \beta_1], \dots, (0, \beta_d]$
be real numbers satisfying~\eqref{eq : alpha + beta > d}.  
Set $\gamma' = d - \sum_{i=1}^d \beta_i'.$  
Then the distribution $\fdgd$ belongs to the space $\C^{-\gamma'}_{\loc}(\R^d)$.  
Moreover, there exists a constant $C > 0$ such that, for every $n \in \N$, we have
\[
\|\fdgd\|_{\C^{-\gamma'}(\nn)}
\le C\, \|f\|_{\C^{\alpha'}([-n, n + N]^d)}
   \prod_{i=1}^d \|g^i\|_{\beta_i', [-n, n + N]^d}.
\]
\end{prop}

\begin{proof}
Since the proof is similar to that of Proposition~\ref{prop : régularité de la distribution globalement höldérien},
we only recall the inequalities that differ.  
Recall the notations
\[
\beta' = \sum_{i=1}^{d} \beta_i'
\quad \text{and} \quad
K = N^{\alpha' + \beta'}
    \max_{i = 1, \dots, 2^d - 1}
    \left(
      \|\phi\|_{\C^{\alpha'}(\R^d)},
      \|\psi^{(i)}\|_{\C^{\alpha'}(\R^d)}
    \right).
\]

Let $n \in \N$.  
We shall use the characterization~\eqref{eq : distribution localement hölder}.  
Let $i \in \{0, \dots, 2^d - 1\}$, $j \in \N$, and $k \in \Z^d$ such that $|k 2^{-j}| \le n$.  
We have the inclusion
\begin{equation}
\label{eq : preuve régularité locale inclusion}
\kNj \subset [-n, n + N]^d.
\end{equation}

Using the same argument as in the proof of Proposition~\ref{prop : régularité de la distribution globalement höldérien}
and the above inclusion of supports, we deduce from the sewing inequality~\eqref{eq : borne de couture dans le cas localement hölder} that
\[
|\cijk|
\le C\, K\, 2^{\gamma' j}
   \|f\|_{\C^{\alpha'}([-n, n + N]^d)}
   \prod_{m=1}^d \|g^m\|_{\beta_m', [-n, n + N]^d},
\]
where the constant $C > 0$ depends only on $d, \alpha', \beta_1', \dots, \beta_d'$.  
A similar reasoning yields the inequality
\[
|\ck| \le C\, K,
\]
which completes the proof of the proposition.
\end{proof}

At this stage, we can define the mapping $\M$ of Theorem~\ref{theo : distribution locale f dg1 ^....^ dgd}:
\[
\begin{array}{rrcl}
\M :  & \Calphaloc(\R^d) \times \C^{\beta_1}_{\loc}(\R^d) \times \dots \times \C^{\beta_d}_{\loc}(\R^d)
& \longrightarrow & \C^{-\gamma}_{\loc}(\R^d)  \\
& (f, g^1, \dots, g^d) & \longmapsto & \fdgd.
\end{array}
\]
We have shown that this mapping satisfies properties~(1) and~(2) of Theorem~\ref{theo : distribution locale f dg1 ^....^ dgd}.  
Moreover, by definition, it coincides with the mapping introduced in the framework of Theorem~\ref{theo : distribution globale f dg1 ^....^ dgd} when the functions $f, g^1, \dots, g^d$ are globally Hölder.  
The following lemma is a local version of the approximation lemma~\ref{lem : lemme d'approximation global}.  
It allows us, using the same arguments as in the globally Hölder case, to prove the uniqueness of the mapping~$\M$.

\begin{lemme}
\label{lem : lemme d'approximation local}
Let $\alpha, \alpha' \in (0,1)$ with $\alpha' < \alpha.$  
For every function $f \in \Calphaloc(\R^d),$ there exists a sequence of functions $f_n : \R^d \rightarrow \R, \; n\in\N,$ satisfying the following properties.
\begin{enumerate}[label=\arabic*.]
\item For every integer $n\in \N^{\ast},$ the function $f_n$ is of class $\C^{\infty}$ on $\R^d,$ and for every compact set $K \subset \R^d,$ we have
$$ \|f_n\|_{\Calpha(K)} \le \|f\|_{\Calpha(K_1)},$$
where $K_1 = K + \overline{\B}(0,1).$
\item For every compact set $K \subset \R^d,$ 
$$ \lim_{n \to \infty} \|f-f_n\|_{\C^{\alpha'}(K)} = 0.$$
\end{enumerate}
\end{lemme}

We can now define Züst’s integral in this setting.  
Let $\Omega \in \B(\R^d)$ be a bounded Borel set such that the function $\mathbb{1}_{\Omega}$ belongs to the space $\B^{\gamma}_{1,1}(\R^d)$.  

We consider a function $\psi = \psi_{\Omega} \in \D(\R^d)$ satisfying the following properties:
\begin{enumerate}[label=\arabic*.]
\item $\psi = 1$ on a ball $\B_1$ containing $\Omega$;
\item $\psi = 0$ outside a ball $\B_2$ strictly containing $\B_1$.
\end{enumerate}
The distribution $\fdgd \cdot \psi$ then belongs to the space $\C^{-\gamma}(\R^d)$, and we define
\begin{equation}
\label{eq : Int_Omega fdgd = < fdgd, 1_Omega> cas local}
\int_{\Omega} \fdgd = \langle \fdgd \cdot \psi, \mathbb{1}_{\Omega} \rangle,
\end{equation}
which does not depend on the choice of the function~$\psi$.

\bigskip

\noindent\textbf{Acknowledgements.} The author thanks his advisors Thierry Lévy and Lorenzo Zambotti for posing the questions considered in this paper, for many enlightening discussions and for their constant encouragements. The author also thanks Elias Nohra and Stéphane Seuret for helpful remarks on the redaction of this paper.

\nocite{*}
\bibliographystyle{plain}
\bibliography{biblio}
\end{document}